\documentclass[letterpaper,10pt,reqno]{amsart}
\usepackage[usenames,dvipsnames]{xcolor}

\RequirePackage[OT1]{fontenc}
\usepackage{underscore}
\usepackage[portrait,margin=3.5cm]{geometry}
\usepackage{mathrsfs}
\usepackage[colorlinks,citecolor=blue,urlcolor=blue,linkcolor=blue,linktocpage=true]{hyperref}
\usepackage[foot]{amsaddr}
\usepackage{amssymb,amsthm,amsfonts,amsbsy,latexsym,dsfont}
\usepackage[textsize=small]{todonotes}       
\usepackage{graphicx}
\usepackage[english]{babel}
\usepackage{subfigure}
\usepackage[title, header]{appendix}
\usepackage{enumerate}

\usepackage{color}              
\usepackage{graphicx}
\usepackage[]{amsmath}
\usepackage{amsthm}
\usepackage{amssymb}
\usepackage{bbm}
\usepackage{hyperref}
\usepackage{comment}
\usepackage{microtype}

\definecolor{MyDarkblue}{rgb}{0,0.08,0.50}
\definecolor{Brickred}{rgb}{0.65,0.08,0}

\hypersetup{
colorlinks=true,       
    linkcolor=MyDarkblue,          
    citecolor=Brickred,        
    filecolor=red,      
    urlcolor=cyan           
}

\newtheorem{theorem}{Theorem}[section]
\newtheorem{lemma}[theorem]{Lemma}
\newtheorem{proposition}[theorem]{Proposition}
\newtheorem{corollary}[theorem]{Corollary}

\newtheorem{definition}[theorem]{Definition}

\newtheorem{remark}[theorem]{Remark}
\newtheorem{claim}[theorem]{Claim}
\newtheorem{example}[theorem]{Example}

\newcommand{\Pv}{\mathbb{P}}

\newcommand{\Ev}{\mathbb{E}}

\newcommand{\e}{{\mathrm e}}

\numberwithin{equation}{section}

\newcommand{\R}{\mathbb{R}}
\newcommand{\N}{\mathbb{N}}

\renewcommand{\emptyset}{\varnothing}

\newcommand{\CD}{\mathcal {D}}

\newcommand{\CN}{\mathcal {N}}

\newcommand*{\la}{\lambda}

\newcommand*{\de}{\delta}

\newcommand*{\ve}{\varepsilon}

\newcommand*{\al}{\alpha}

\newcommand*{\be}{\begin{equation}}
\newcommand*{\ee}{\end{equation}}
\newcommand*{\ba}{\begin{aligned}}
\newcommand*{\ea}{\end{aligned}}
\newcommand*{\barr}{\begin{array}{c}}
\newcommand*{\earr}{\end{array}}

\def \toas     {\buildrel {a.s.}\over{\longrightarrow}}

\newcommand*{\wit}{\widetilde}

\newcommand*{\ind}{\mathbbm{1}}
\def\namedlabel#1#2{\begingroup
    #2%
    \def\@currentlabel{#2}%
    \phantomsection\label{#1}\endgroup
}

\begin{document}
	\title[Explosive CMJ branching processes]{Explosive Crump-Mode-Jagers branching processes}

	\date{\today}
\author[J. Komj\'athy]{J\'ulia Komj\'athy}
	\address{Department of Mathematics and
	    Computer Science, Eindhoven University of Technology, P.O.\ Box 513,
	    5600 MB Eindhoven, The Netherlands.}
	\email{j.komjathy@tue.nl}

\begin{abstract}
In this paper we initiate the theory of Crump-Mode-Jagers branching processes (BP) in the setting where no Malthusian parameter exist, i.e., the process grows faster than exponential. A Crump-Mode-Jagers BP is a branching process (in continuous time) where arbitrary dependencies are allowed between the birth-times of the children of a single individual in the population. It is however assumed that these reproduction processes are i.i.d. point processes for different individuals. This paper focuses on determining whether this branching process explodes, that is, the process reaches infinitely many individuals in finite time. We develop comparison techniques between reproduction processes.
We study special cases in terms of explosivity such as age-dependent BPs, and epidemic models with contagious intervals. 
For this, we superimpose a random contagious interval $[I, C]$ on every individual in the BP and keep only the children with birth-times that fall in this interval of the parent. We show that the distribution of the end $C$ of the contagious interval does not matter in terms of explosion, while the distribution of $I$ does: the epidemic explodes if and only if the two age-dependent BPs with the original birth-times and birth-times $I$ explode. We finish studying some pathological examples such as birth-time distributions that are singular to the Lebesque-measure yet they produce an explosive BP with arbitrary power-law offspring distributions. 
\end{abstract}

\maketitle
\section{Introduction and model}
\subsection{Introduction} The study of branching processes has a long history in probability theory, see e.g.\ the classical books \cite{AN72, Har63} for a good introduction to the field.
 
In this paper we study branching processes with general reproduction functions, as described e.g.\ in \cite{jagers1974convergence, jagers1984growth} or \cite{nerman1981convergence}. This means that every individual in the population reproduces according to an i.i.d.\ point process, and no particular assumption is made about the dependencies between the times of arrival of points (the birth-times of consecutive children) of a single individual.
We study these branching processes from the point of view of explosion event: we investigate when is it possible that infinitely many individuals are born within a finite time.

It is well-known that if the expected number of children of an individual is finite, then the population grows exponentially time, and the growth rate is called the \emph{Malthusian parameter}. That is, at time $t$, the population size will be of order $\exp\{\la t \},$ with $\la$ being the Malthusian parameter. The concept of a Malthusian parameter is wider than finite mean offspring: for a precise definition, see \eqref{eq::mal-def} below.
Under the assumption that a Malthusian parameter exists, the population grows exponentially in time \cite{jagers1984growth}, thus explosion does not occur. The behavior of these branching processes is quite well understood,  due to the immense work in the `70s and `80s.
In particular, the age-distribution, remaining lifetime distribution, structure of generations and relationships and other general characteristics were studied, see e.g. the work of Jagers, Nerman, B\"uhler and Samuels, and others, e.g. \cite{berndtsson1979exponential, Bu71, Bu72, Bu74,  jagers1974convergence, jagers1984growth, nerman1981convergence, Sa71} for further references.

In this paper we focus on the case when the BP grows \emph{faster} than exponential. 
 For explosion to happen, an infinite mean offspring is required, but is not sufficient. 
Even when the number of offspring of a single individual is infinite, the process still might grow exponentially if the birth-times are quite spread out: A Malthusian parameter might still exist,  e.g. in the case when each individual reproduces infinitely many times following a homogeneous Poisson process. 

The study of infinite mean BP-s that do not have a Malthusian parameter is rather limited: as far as the author is aware of, the literature is restricted to two main directions:
Firstly, Galton-Watson branching processes with infinite mean offspring. This is the case when every individual produces all its children after a unit time. 
The limiting behaviour of these processes has gained reasonable attention in the `70s and `80s, see e.g. \cite{Dav78, BarSch77} for further literature. Here, results investigate the (double-exponential) growth of the process and study properties of the limiting variable when some slowly varying function of the number of individuals in the $n$th generation is taken. See \cite{BarSch77} for further references. 

The second direction that gained sufficient attention is the explosion of age-dependent branching processes, with infinite mean offspring. An age-dependent BP is a process where each individual lives an i.i.d. lifetime from some distribution and they produce all their children upon death.

The question of explosivity of age-dependent BPs were studied first in the Russian literature: Sevastanov gave sufficient criterions for explosion, see e.g. \cite{Sev70, Sev67, Vat87}. A similar paper in flavour is by Grey \cite{Grey74} to the west of the iron curtain. Later, the study of (minimal) displacement of branching random walks rediscovered the same topic and new results were found: 
the linear growth of minimal displacement were established for processes with finite mean offspring in \cite{BIg90, Ham74, Kin75}. Then, 
the minimal displacement of branching RWs with possible $0$ replacement were studied by \cite{bramson78, DeHo91} and finally \cite{AmiDev13} provided necessary and sufficient criterions for the explosivity of age-dependent BPs under mild conditions on the offspring distribution. 
 
Other generalisations include the topic branching Markov processes, see Savits \cite{Sav69}, and branching processes were a single individual can produce infinite offspring and thus `terminate' the process, see the recent work by Sagitov \emph{et al.} \cite{Sag15, SagLin15}.


From the more applied point of view, the study of \emph{complex networks} and \emph{random graph models} has boomed in the last few decades.
Most random graph models locally have a tree-like structure and as result can be well-approximated by branching processes: the Erd{\H o}s-R\'enyi random graph \cite{ER60}, the configuration model \cite{Bollobas01}, inhomogeneous random graphs \cite{BJR}, preferential attachment models \cite{Bara99} all fall into this category. 
The first step in the study of the topology of these graph models as well as of the behavior of dynamical processes on them (such as information diffusion, epidemic spreads, etc), is to understand how the graph locally looks like and how the process on it under investigation behaves locally. This local behavior is well-approximated by an appropriately chosen branching process.  It is thus crucial to understand the behavior of the approximating branching process. 

To give a list of examples where branching processes are used in the analysis of random graphs, we start with the classical example: the phase transition in the size of the giant component in the Erd{\H o}s-R\'enyi random graph corresponds to the sub/super/criticality of the corresponding Galton-Watson BP \cite{Bollobas01, H10}.  The same is true for the configuration model \cite{MolRee95, MolRee98} and for inhomogeneous random graphs \cite{BJR, HHZ07}. 

In weighted random graphs, typical (weighted) distances can be studied by using age-dependent branching processes where the birth-time distribution equals the edge-weight distribution. For typical distances, one needs to understand the growth rate of the branching process as well as the behavior of certain characteristics, such as the ratio of dead and alive particles, and the \emph{age distribution} of the currently alive population \cite{berndtsson1979exponential, jagers1974convergence}. Results on these are utilized e.g.\ in \cite{BHH10, BHH11, BHH14, HHM05, HHZ07}.
Further, the number of edges on the least weight path, the \emph{hopcount} between two vertices, depends on the asymptotic behavior of the generation of the $k$-th born particle in the corresponding BP, thus the work of B\"uhler \cite{Bu71, Bu72, Bu74} and Kharlamov \cite{Khar1} is used in e.g. \cite{BHH10, BHH11, BHH14,  HHM05, KolKom12, KomVad15}. Typical distances in weighted random graphs also correspond to the transmission time for an information or for an epidemic to spread between two vertices: again, branching process results are extensively used to study the behavior of epidemics, 
as in \cite{BaRe13, bhahofkom}. 

A large number of real life networks are known to have power-law degree distributions with power-law exponent in the interval $(2,3)$. 
This corresponds to finite asymptotic mean but infinite asymptotic variance for the empirical degree distribution. 
Examples include the internet on the router level \cite{falo99}, the world wide web \cite{bara00}, cargo ship movements \cite{kalu10}, gene regulatory networks \cite{born02}, citation networks \cite{radi12}, and many more, see more references in \cite{bara00, newman03}. Due to the size-biasing effect (a neighbor vertex in the graph is more likely to have more neighbors), a power-law exponent between $(2,3)$ for the degrees in the graph corresponds to an approximating branching process with power-law offspring distribution with exponent in $(1,2)$: that is, the approximating BP has infinite mean offspring. 
For the study of the random graph models of the above mentioned networks, it is thus crucial to understand the behavior of these branching processes. 
For example, typical distances and the behavior of spreading processes with deterministic transmission times in the configuration model can be determined in this regime using the results of Davies \cite{Dav78}: see the articles \cite{BarHofKom15, HHZ07, HofKom15} or the book \cite{H10}.

If we would like to understand the topology of weighted networks with infinite asymptotic variance degrees, we are in the regime of this paper: results on age-dependent or CMJ branching processes with infinite mean offspring are needed. Since the literature is limited to the explosive case \cite{Grey74, Sev70, Sev67}, or Galton-Watson BPs \cite{Dav78}, most results in this setting are also limited to either explosive propagation or deterministic edge-weights \cite{BarHofKom16}.

For preferential attachment graphs, CMJ branching processes is probably the most natural way to look at the local neighborhood of a vertex, see e.g. the work of Bhamidi \cite{Bha07}, T\'oth and Rudas \emph{et al} \cite{Rudas08, Rudas07}, Dereich and M\"orters  \cite{Der16, DerMor16}.

We can conclude one important fact from all these examples: limitations on the knowledge on branching processes result in limitations on our understanding of spreading processes on random graphs. For example, little to almost nothing is known about nondeterministic but non-explosive information diffusion on random graphs with infinite variance degrees \cite{BarHofKom16}, and literally nothing is known when we would like to assume dependencies between the transmission times from a vertex to its neighbors. This might easily occur when for instance we investigate an epidemic where each individual has an incubation time: to determine whether an epidemic is explosive is of extreme importance, think of for instance the slow spreading but extremely deadly recent case of Ebola \cite{ebola} or the latest news about the Zika virus \cite{zika}.

This paper aims to establish the foundations of the research of this missing area, and is thus part of a long-term project: the recent interest in spreading processes on random graph models make it necessary to extend the current theory on Crump-Mode-Jagers branching processes to the infinite mean offspring case. 

\subsection{Summary of results} This paper is divided into sections where each section is devoted to one topic and results are presented within the section. To give an overview,  we state here informally our results without specifying the detailed conditions on the models. 

First, in Section \ref{s:model} we introduce the model and study BPs with general reproduction function: 
Let us assume that the children of the initial ancestor are born at times $\sigma_i, i\in \N$ (if there are finitely many children, we set the birth-times infinite after the last child is born). We show that explosion can only happen via having infinite rays (line of descendants) with finite total length, see Lemma \ref{lemma::rays} below. 

In Section \ref{s::op},   we study the operator $T_\xi$ acting on non-increasing functions with values in $[0,1]$, corresponding to the  the distributional identity
\[ V\  {\buildrel d \over = }\ \min_{i\in \N} V_i+ \sigma_i, \]
that is ensured by the basic  branching structure of the process.
We show that the process explodes if and only if there is a non-constant fixed point function of this operator, 
and study the properties of the distribution function of the explosion time. In particular we show that 

1) the distribution function of the explosion time is the smallest fixed point function of the operator (Lemma \ref{lem::fixpoint}) and it is non-decreasing (Claim \ref{cl::noninc}),

2) explosion can happen arbitrarily fast (Claim \ref{cl::fastexp}),

3)  the BP a.s. explodes on survival (no conservative survival is possible) (Claim \ref{cl::nocons}).

4) If two reproduction functions have corresponding operators with one operator dominating the other, then the explosion of the process with dominating operator implies the explosion of the other process (Claim \ref{cl::opdom}).

In Section \ref{s::3} we study BPs with general reproduction processes. In particular, we investigate what happens if $\xi$ has a positive expected mass at $0$, or has finite expected mass in some interval around $0$, see Theorem \ref{thm::mass-zero}. In particular, Theorem \ref{thm::mass-zero} implies that a process can never be explosive if the reproduction process has finite expected mass on compact intervals. 
In this section, we develop a method that we call \emph{stochastic domination around the origin} and \emph{coupling around the origin}, respectively, see Definitions \ref{def::0-dom} and \ref{def::point-proc-dom}. We prove  a general comparison theorem stating that if for some $t_0>0$, two point processes can be coupled in such a way that one of them has at least as many points on every interval $[0,t]$, for all $t<t_0$ than the other process has, then the explosion of the latter process implies explosion of the first process, see Theorem \ref{thm::comparison-general}.  

In Section \ref{s::classic}, we introduce classical examples such as 

1) age-dependent branching processes, where each individual has a random number of children with i.i.d.\ birth-times from distribution $\sigma$.

2) epidemic models with incubation times and/or contagious periods, where each individual has a random number of children with i.i.d.\ birth times from distribution $\sigma$, but only those are actually born who fall within a random interval $[I, C]$. $I$ is called the \emph{incubation time} while $C$ is called the end of the \emph{contagious period}.

3) Backward version of epidemic models, where each individual has a random number of children with i.i.d.\ birth times from distribution $\sigma$, but child $i$ is only born if it falls within a random interval $[I_i, C_i]$, where the intervals $[I_i, C_i]$ are independent copies of $[I, C]$.

Then, in Section \ref{s::comp-class} we develop very natural comparison theorems saying that the following help the explosion to occur:

1) shorter birth-times $\sigma$ (without incubation times: Theorem \ref{thm::comp1}, with incubation times:

\quad Theorem \ref{thm::comp4}),

2) longer contagious period $C$ (Theorem \ref{thm::comp1}),

3) shorter incubation times $I$ (Theorem \ref{thm::comp2}),

4) more offspring (Theorem \ref{thm::comp-tail}).

Further, we show that the explosion of the backward process always implies the explosion of the forward process, see Theorem \ref{thm::forw-back}. 

In Section \ref{s::effect1}, we start to investigate the effect of contagious periods and incubation times in more detail.
We show that the distribution of $C$, the end of the contagious period, \emph{does not matter} in terms of explosion: it is impossible to stop the explosion by superimposing a contagious period $C$ on an explosive process (with or without incubation times), see Theorem \ref{thm::cont-nomatter2} and Corollary \ref{thm::cont-nomatter}. 

On the other hand, we show that incubation times \emph{do matter} in terms of explosion: any explosive age-dependent BP becomes conservative when we superimpose an incubation time $I$ on it that would constitute a conservative process if it were the birth-time distribution, see Theorem \ref{thm::incu-matter}. Thus, explosivity of the two age-dependent BPs with birth-times $\sigma$ and $I$, respectively, is necessary for the epidemic model with incubation times $I$ to be explosive. 

In Section \ref{s:minsum}, we introduce the notion of minimum-summability,  (\emph{min-summability} in short): a distribution $\sigma$ is min-summable for an infinite sequence $(a_k)_{k\in \N}$, if, taking the minimum of $a_k$ i.i.d.\ copies of $\sigma$ is summable in $k$ almost surely. This notion was introduced by 
Amini \emph{et al.}\ in \cite{AmiDev13}, where they showed that min-summability of distribution $\sigma$ with respect to the generation sizes of a Galton-Watson BP ($a_k=Z_k$) implies the explosivity of the age-dependent BP with birth times $\sigma$ and offspring as in the Galton-Watson BP, see Theorems \ref{thm::amidev1} and \ref{thm::amidev2}, as long as the offspring distribution has sufficiently heavy tails, that they call \emph{plump} (Definition \ref{def::plump1}).

We introduce the notion of \emph{plump power-law} distributions, see Definition \ref{def::plump2}. Roughly speaking, a distribution is a plump power-law if the tail behaves as a power-law, where the exponent can vary infinitely many times between different values, but it can not be much heavier than that, at least eventually. We give a rather transparent integral condition that the distribution of the birth-times $\sigma$ should satisfy that is necessary and sufficient for the explosivity of \emph{all} age-dependent BPs with plump power-law offspring, see Lemma \ref{lemma::integral}. This in particular implies that if a birth-time distribution is explosive for one offspring distribution with power-law exponent $\al \in (0,1)$, then it is explosive for \emph{all} offspring distribution that are plump power-laws, see Corollary \ref{cor::power-law-change}.

In Theorem \ref{thm::max-birthtime} we show that explosivity is closed under the following operations on the birth-time distribution:

1) multiplication by a positive constant

2) taking the maximum or minimum of two or more independent variables

3) taking the sum of two or more independent variables

4) binomial thinning: each individual is kept only with some fixed probability $p$.

In Section \ref{s::newproof}, we provide a new proof of the harder direction of Theorems \ref{thm::amidev1} by \cite{AmiDev13}: we show that for plump power-law offspring distributions, min-summability implies explosion. We do this by a generation-dependent thinning approach of the BP. In each generation, we throw away all the edges that have too long birth-times so that each remaining infinite ray has summable total length and show that the thinned process is supercritical. This proof is important since it reveals the robustness of  explosivity in the choice of the power-law exponent $\alpha$. Further, it can be adapted to analyse the explosivity of epidemic models with incubation periods, where independence is lost between the birth-times of the children of the same individual. 

In Section \ref{s::incu}, we investigate epidemic models with incubation times. In Theorem \ref{thm::incu-matter}, we have shown that an explosive process can be stopped by superimposing an incubation time on it that would form a conservative process if it would be used as birth-time distribution. This raises the natural question: can an \emph{explosive} incubation time distribution stop an age-dependent BP to be explosive? In other words: Assume that two branching processes with the same offspring distribution $X$, birth-time distributions $\sigma$ and $I$, respectively, are explosive. Is then the epidemic model with offspring distribution $X$, birth-time distribution $\sigma$ and incubation time distribution $I$ always explosive?

The answer turns out to be \emph{yes} but the current proof is far from trivial and does not cover the full generality. We prove first in Theorem \ref{thm::incu-main} that if $X$ is a plump distribution that forms an explosive BP with both possible birth-time distributions $\sigma$ and $I$, then the combination when one of them is used as incubation time produces an \emph{explosive backward-process of epidemics}. The advantage of the backward process is that it maintains independence across the birth-times of children of an individual. However, since only those edges are kept that are longer than the incubation time on that edge, exactly the \emph{long edges} are likely to be kept and hence explosivity of the thinned process is not trivial. Somewhat surprisingly, the bad cases are those birth-times that have a very steep distribution function around the origin: these are likely to be small, and hence thinned by an incubation time. The proof goes by verifying the integrability condition developed earlier for the new (thinned) birth-time distribution.

Then, in Theorem \ref{thm::incu-main-forw}, we prove the same statement for the forward process: namely that two explosive age-dependent BPs, when combined as one of the birth-times serving as incubation time, always forms an explosive \emph{forward process} of an epidemic. The proof is similar to the new proof of the min-summability Theorem \ref{thm::amidev-re} : here, we perform a generation dependent thinning of vertices: we thin a vertex in generation $n$ if either its incubation time is longer than $\delta t_n$, or its birth-time does not fall in the interval $(\delta t_{n-1}, t_{n-1}]$. We choose the thinning thresholds $t_n$  to be summable in an appropriate way. This thinning ensures that if an infinite ray is not thinned, then all the vertices on it satisfy that their birth-time is larger than the incubation time of their parent vertex, i.e., the forward process of the epidemic can proceed to infinity on any such ray in finite total time. We show that for an appropriate choice of $t_n$, the thinned process forms a supercritical process, whenever the offspring distribution is a plump power-law.  

The combination of Theorems \ref{thm::incu-main} and \ref{thm::incu-main-forw} shows that for \emph{plump power-law} offspring distributions, the explosivity of the two age-dependent BPs is \emph{necessary and sufficient}  for the explosivity of the epidemic model with incubation times (both backward and forward versions). For strictly plump distributions, (i.e.\ those that have heavier tail than power-law, infinitely often, e.g.\ $1-F(t)=C/\log t$), the question of sufficiency remains open for the forward process. The backward process is covered in Theorem \ref{thm::incu-main} for this case as well. The author conjectures that this is indeed the case.

Finally, in Section \ref{s::edu} we study some interesting examples of birth-time distributions: we show e.g. that singular distributions can easily form an explosive BP with power-law offspring distributions: we go as far as constructing a distribution function that allocates discreet masses on a sequence of points, that have $0$ as their only accumulation point.

\subsection{Open problems}
This paper is essentially only concerned with the question whether a CMJ branching process explodes, and thus can be considered as the first step in the analysis of these processes. From the random graph points of view, there is a need to investigate other characteristics of these processes as well: 
the \emph{generation of the $k$th born individual}, for instance, is a relevant question in understanding the hopcount between two vertices in the graph. Finer properties of the distribution of the explosion time could be also studied such as its \emph{tail-behaviour} and its \emph{behaviour around the origin}. 

Further, for \emph{conservative} BPs with infinite mean offspring, most questions are open, even for the `simple' age-dependent case, when all the individuals have independent birth-times. Here we distuinguish two cases: when the support of the distribution of birth-times goes down to $0$ or when it does not, i.e., when $\inf\mathrm{supp}\, F=0$ or when $\inf\mathrm{supp}\, F=c>0$. 

In the first case, from the random graph point of view, it is an urgent need to study the possible asymptotic behavior of the \emph{date of birth of the $k$th individual} in the process: in principle this can be any function that grows to infinity not faster than of order $\log\log k$. For the time to reach an individual in generation $k$, the asymptotic behavior is known under some conditions, see the work of Bramson \cite{bramson78} and Dekking and Host \cite{DeHo91}. The generation of the $kth$ born individual in this setting is another relevant quantity that needs further investigations.

When the support of the birth-time distribution is strictly above $0$, the Dirac-delta case is relatively well-understood: we are in the setting of an infinite mean Galton-Watson BP, see again \cite{Dav78, BarSch77} for further references. If the birth-times are of the form $\sigma= c + X$ for some random variable $X$, it seems to the author that the distribution of the additional birth-time $X$ must play an important role in the growth of the process: when this extra time would form an explosive process, we expect that the behaviour of the process will be to a large extent similar to that of the Galton-Watson case. The behaviour of the process when the extra time $X$ would form a conservative BP is for now unknown.

The behavior of conservative general CMJ branching processes is another interesting topic. Some recent work in this direction is that of Dereich \emph{et al}\cite{Der16, DerMor16}, who investigate condensation phenomenon in certain CMJ processes where no Malthusian parameter exists.

 \subsection{The model}\label{s:model}
A general Crump-Mode-Jagers type branching process (BP) is constructed as follows: an initial ancestor, $0$, the root of the process, is born at time zero. She is a mother\footnote{Jagers and Nerman \cite{jagers1984growth} used this wording. We find it quite natural, so we will stick to mothers and daughters.} of some children, and her reproduction process, denoted by $\xi_0$, contains the consecutive times of birth. We assume that $\xi_0$ is a copy of a point process $\xi$. We denote by $0\le \sigma_1\le \sigma_2\le \dots$ the consecutive appearance of points in $\xi$. We only assume for now that the total number of points in $\xi$ is countable almost surely. If it is finite and equals $k$, we set $\sigma_i=\infty$  for all $i>k$.  Then for $0\le s\le t\le \infty$ we define
\[ \xi[s,t]:=\sum_{i\in \N} \delta_{\sigma_i} \ind_{\sigma_i \in [s,t]}.\]
We abbreviate $\xi(t):=\xi[0,t]$. Note that $\xi[s,t]$ is a measure valued variable. We write $|\xi[s,t]|, |\xi(t)|$ for the total mass in the interval $[s,t]$ and $[0,t]$, respectively.

Each of the children of the root, born on date $\tau_i:=\sigma_i$, have their own `life story' given by reproduction processes $(\xi_{i})_{i\in \N}$ that start at date $\tau_i$, and conditioned on $\tau_i$, they are i.i.d.\ copies of the point process $\xi$. More generally, each descendant $x$ of the root reproduces in an i.i.d.\ manner, following a copy of the point process $\xi$, shifted to start at the descendant's birth date $\tau_x$. 
We can code the descendants in generation $n$ by words from the alphabet $\N^+$: an individual $x=i_1i_2 \dots i_n$ is the $i_n$th child of the $i_{n-1}$th child of the ... of the $i_1$th child of the initial ancestor, and her date of birth is $\tau_x=\tau_{i_1i_2\dots i_{n-1}}+\sigma_{i_n}^{(i_1i_2\dots i_{n-1})}$, where  $\sigma_{i_n}^{(i_1i_2\dots i_{n-1})}$ denotes the time of appearance of the $i_n$th point in the process $\xi_{i_1i_2\dots i_{n-1}}$.  We denote the set of all words of length $n$ by $G_n$, and the set of all possible words by $G$, i.e.,
\[ G_n:=\N^n, \quad G:= \{0\} \bigcup_{n\in \N} G_n.\]
We denote by $G_\infty$  the set of words of infinite length, i.e., an infinite sequence $i_1 i_2 \dots$, and we call these \emph{rays}. For a word $x\in G_n,\  n>k$, we denote $x_{|k}$ the truncation of $x$ at length $k$, that is, if $x=i_1i_2\dots i_n$ then $x_{|k}=i_1i_2\dots i_k$. This is the ancestor of $x$ in generation $k$. Similarly, we write $x_{|-k}$ for the $k$-th ancestor of the individual $x$, i.e., again for $x=i_1i_2\dots i_n$, $x_{|-k}=i_1i_2\dots i_{n-k}$. Note that $x_{|-1}$ is the mother of $x$.
In the sequel, we refer to $\tau_x$ as the \emph{date of birth} of $x$, and somewhat misusing the notation, $\sigma_x$ as the \emph{birth-time} of $x$, i.e.,  $\sigma_x:=\tau_x-\tau_{x_{|-1}}$ describes how old was the mother of $x$ when she gave birth to $x$. 
With these notation in mind, at time $t$, the already born children of individual $x$ are given by
\[ \xi_x(t-\tau_x),\]
where we define $\xi(t):=\emptyset$ for $t<0$.
We observe the already existing population at time $t$: \[ \CD(t):=\bigcup_{x\in G}\ \ind_{\tau_x\le t}. \]
We set $D(t):=|\CD(t)|$. 

\emph{The coming generation}. In the theory of CMJ branching processes, the coming generation at time $t$, $\CN(t)$, plays a crucial role: these are the individuals who are not yet born but whose mother is already born. Or, the other way round, the future children of already alive individuals, i.e., 
\[ \CN(t)=\bigcup_{x\in G} \xi_x[t-\tau_x, \infty] \ind_{\tau_x\le t}.  \]
We set $N(t):=|\CN(t)|$. We comment here on the naming differences: in many papers $\CD(t)$ is called the set of \emph{dead individuals} while $\CN(t)$ is called the set of \emph{alive individuals} at time $t$, however, when these names are used, then a mother reproduces upon death and gives rise to its chidren.
In this naming, the set $\CN(t)$ is already \emph{part of the branching process at time $t$}. Here we stick to the more mild naming tradition and suppose that mothers give birth during their life process $\xi$ (and we do not assume anything about their death).

\emph{Explosion of the process}. The branching process is called explosive if the event of reaching infinitely many individuals in finite time has positive probability. 
More precisely, 
\begin{definition}[Explosive vs.\ conservative BPs]
A branching process with reproduction function $\xi$ is called explosive if for some $t>0$
\[\Pv( D(t)=\infty) >0.\]
Otherwise it is called conservative.
\end{definition}
Clearly, if for some $t>0, \ \Pv(\xi(t)=\infty)>0$, then the process is explosive, since a single mother can produce infinitely many children. We call this sideway-explosion and we are not concerned with it in this note. We refer the interested reader to the work of Sagitov \emph{et al.} \cite{Sag15, SagLin15} and references therein. 
\begin{definition}Let $M_n:=\inf\{t: \CD(t) \cap G_n \neq 0 \}$ denote the first time an individual in $G_n$ is born, i.e., $M_n$ is the time to reach generation $n$. Let $\tau^{n}$ denote the date of birth of the $n$th appearing new individual in the population, i.e., $\tau^n=\inf\{t: D(t)=n+1\}$ is the time to reach size $n+1$ for the population.
\end{definition}
Note that $M_\infty$ is the length of the shortest ray to infinity, while $\tau^\infty$ is the explosion time, the first time when $D(t)=\infty$. Here we arrive at our first lemma. 

\begin{lemma}[Explosion\,=\,convergent rays]\label{lemma::rays}
Assume that almost surely for all $t>0$, $\xi(t)<\infty$ holds. Then, 
\be\label{eq::tau-lim} \lim_{n\to \infty} \tau^n = \lim_{n\to \infty} M_n:=V,\ee
where $V$ is called the explosion time of the process. 
This means that the event $\{V\le t\}$ is equivalent to having an infinite ray with finite total length, i.e., 
\[  \{ \exists\, x \in G_\infty: \forall n\in \N,  \ \tau_{x_{|n}} \le t \}.\]
If further $\xi(\infty)<\infty$ holds almost surely, then $\{N(t)=\infty\} \subseteq \{D(t) =\infty\}$.
\end{lemma}
\begin{proof}
Note that $\{\lim_{n\to \infty} \tau^n \le t\} = \{D(t)=\infty\}$. We show that this is equivalent to 
\be\label{eq::intersect} \bigcap_{k\in \N} \{ \exists\, x\in G_k, \tau_x\le t\}=\lim_{n\to \infty}\{ |G_k\cap \CD(t)|\ge 1\}=\{\lim_{n\to \infty} M_n\le t\} .\ee
The direction where \eqref{eq::intersect} implies $\{D(t)=\infty\}$ is obvious by noting that the left hand side is a nested sequence. Indeed, if $x \in G_n\cap \CD(t)$ then $x_{|k} \in \CD(t )$  for all $k\le n$, i.e., all the ancestors of $x$ are also already born, hence, $|G_k\cap \CD(t)|>0$.

For the reverse direction, clearly $|G_0\cap \CD(t)|=1<\infty$. Since $\xi(t)<\infty$ for all $t\in \R^+$, with inductive reasoning we also obtain $|G_n \cap \CD(t)|<\infty$  for all $n\in \N$, since
\[  |G_n\cap \CD(t)|=\bigcup_{x\in G_n} \ind_{\tau_x\le t}=   \bigcup_{x\in G_{n-1}\cap \CD(t)} \xi_x(t-\tau_x) <\infty\]
since it is a finite union of finitely many points. Since $D(t)=\sum_{n=1}^\infty |G_n \cap \CD(t)|$
the event $\{D(t)=\infty\}$ implies that $|G_n\cap \CD(t)|>0$ for infinitely many $n$. Again, if $x \in G_n\cap \CD(t)$ then $x_{|k} \in \CD(t )$  for all $k\le n$, i.e., all the ancestors of $x$ are also already born, hence, $|G_k\cap \CD(t)|>0$. As a result, we get that $\{D(t)=\infty\}$ implies \eqref{eq::intersect}.

For the last statement of the lemma, note that $N(t)=\sum_{x\in \CD(t)} \xi_x[t-\tau_x, \infty]$. Since each summand is finite,  $\{N(t)=\infty\}$ implies that the number of summands is infinite. 
\end{proof}

\section{An operator-approach.}\label{s::op}
From now on, we will assume $|\xi(t)|<\infty$ almost surely for all $t<\infty$. Here we generalise the operator approach used in \cite{Grey74, Sev67}.
In what follows, we write a recursive functional equation for the generating function of $D(t)$ and $N(t)$. Note that by the branching property, we can decompose $D(t)$ and $N(t)$ using the immediate children of the root
\[\ba N(t)&= \sum_{i\in \N} \left(N^{(i)}(t-\sigma_i)\ind_{\sigma_i<t} + \ind_{\sigma_i>t}\right) = \int_{0}^t N^{(x)}(t-x) \xi(\mathrm dx) + \int_{t}^\infty 1 \xi(\mathrm dx)\\
D(t)&=1 + \sum_{i \in \N} D^{(i)}(t-\sigma_i)\ind_{\sigma_i<t}=1+\int_{0}^t D^{(x)}(t-x) \xi(\mathrm dx),
\ea\]
where $N^{(i)}, D^{(i)}$ are i.i.d. copies of $N(t), D(t)$, and where $N^{(x)}, D^{(x)}$ abbreviates the fact that we need an i.i.d.\ copy of $N, D$ whenever $\xi$ puts non-zero mass at the point $x$.
Using these definitions, the generating function $G_D(s,t):=\Ev[s^{D(t)}]$ for $s\in [0,1]$ satisfies
\[ \ba G_D(s,t)&=s\Ev\left[\Ev\left[s^{\sum_{i \in \N} D^{(i)}(t-\sigma_i)\ind_{\sigma_i<t}}| \xi \right] \right]\\
&=s \Ev\left[ \prod_{i: \sigma_i<t} G_D(s, t-\sigma_i) \right] \\
&=s \Ev \left[ \exp\left\{ \int_{0}^t \log( G_D(s, t-x) ) \xi(\mathrm dx)  \right\}\right], \ea
\] 
where in the second line we used that given the values of $\sigma_i$, the processes $D^{(i)}$ are independent.
Similarly, the generating function $G_N(s,t):=\Ev[s^{N(t)}]$ for $s\in [0,1]$ satisfies
\[ \ba G_N(s,t)&=\Ev\left[\Ev\left[s^{\sum_{i \in \N} N^{(i)}(t-\sigma_i)\ind_{\sigma_i<t}+ \ind_{\sigma_i>t}}| \xi \right] \right]\\
&= \Ev\left[ \Big(\prod_{i: \sigma_i<t} G_N(s, t-\sigma_i) \Big) s^{\xi([t, \infty])} \right] \\
&= \Ev \left[ \exp\left\{ \int_{0}^t \log( G_N(s, t-x) ) \xi(\mathrm dx)+ \int_{t}^\infty \log s \xi(\mathrm dx) \right\}\right]. \ea
\]
We see that $\phi(t):=G_D(1,t)=\sum_{k=1}^\infty \Pv(D(t)=i)=\Pv(D(t)<\infty)$ satisfies
\be\label{eq::phi-recursion-general} \phi(t)= \Ev \left[ \exp\left\{ \int_{0}^t \log( \phi( t-x) ) \xi(\mathrm dx)  \right\}\right].\ee
Further note that $\phi_N(t)=G_N(1,t)$ satisfies the exact same equation. 
Let us introduce the operator $(T_\xi f)(\cdot)$ acting on functions $f:\R^+\to \R^+$ as follows:
\be\label{eq::txi-def} \big(T_\xi f\big)(t)= \Ev \left[ \exp\left\{ \int_{0}^t \log( f( t-x) ) \xi(\mathrm dx)  \right\}\right],\ee
and set $\Omega_{[0,1]}:=\{f:\R^+\to [0,1]\}$ the space of functions with values in $[0,1]$. Then, $T_{\xi}: \Omega_{[0,1]}\to\Omega_{[0,1]}$. Indeed, for any function with values in $[0,1]$, $\int_{0}^t \log (f(t-x)) \xi(\mathrm dx)<0$, exponentiation and taking expectation yields the statement. Further, the monotonicity of the logarithm function, integration, exponentiation and expectation implies that if for two functions $f,g \in \Omega_{[0,1]}, f(t) \le g(t)$ holds for all $t\ge0$ then also $\big(T_\xi f\big)(t) \le \big(T_\xi g\big)(t)$, that is, $T_\xi$ preserves ordering.
\subsubsection{Probabilistic interpretation}
Let us note that the \emph{explosion time of the branching process} can be written as
\[V:=\inf\{ t: D(t)=\infty\}=\sup\{ t: D(t)<\infty\}\]
Then, $\{V>t\}=\{D(t)<\infty\}$ and hence
\be\label{eq::dist-explosion} \Pv(V\le t) = \Pv(D(t)=\infty) = 1-\phi(t).\ee
As a result, we see that $1-\phi(t)$ is the distribution function of $V$.
Note that \eqref{eq::phi-recursion-general} uses the basic branching property of the process: the fact that conditioned on the birth time of the first generation individuals, the different subtrees are independent. It is clear that the BP reaches infinitely many individuals if and only if one of the branches reaches infinitely many individuals, hence, we obtain the distributional identity
 \[ \ V\  {\buildrel d \over =}\  \min_{i \in \N}\, \{ \sigma_i+V_i \},\]
where $V_i$ are i.i.d. from the same distribution as $V$, and $V_i$ denotes the (possibly infinite) explosion time of the subtree of the $i$th child of the  root. 
As a result, 
\[ \phi(t)= \Pv(V>t)=\Pv(\forall i: V_i+ \sigma_i >t)= \Ev\left[ \prod_{i: \sigma_i\le t} \phi(t-\sigma_i) \right]= \big( T_\xi \phi \big) (t).\]

In the coming section, we describe some basic methods of determining whether a process explodes or not by analysing the operator $T_\xi$ more carefully. We also obtain some further properties of explosion times. 
\subsection{Properties of the operator $T_\xi$ and the explosion time}
\begin{lemma}\label{lem::fixpoint} The function $\phi(t):=\Pv(D(t)< \infty)$ is the smallest function that solves the fixed point equation
\be\label{eq::fixpoint} \phi(t)=\big(T_\xi \phi\big)(t), \ee
in the sense that for any other function $f(t)\in \Omega_{[0,1]}$ with $f=T_\xi f$, the inequality $\phi(t)\le f(t)$ holds for all $t\ge 0$. Hence, the branching process with reproduction process $\xi$ is explosive if and only if there exists a fixed point function $\phi(t)\not \equiv 1\in \Omega_{[0,1]}$ that solves \eqref{eq::fixpoint}.
\end{lemma}
\begin{proof}
For the special case given in Example \ref{ex::agedep} below, it was observed in \cite{Grey74, Sev67} that for the following sequence of functions
\be\label{eq::phi-def} \phi_0\equiv 0, \quad \phi_k(t):=\big( T_\xi \phi_{k-1}\big) (t), \ee
$\phi_k(t)$ converges pointwise to $\phi(t)=\Pv(D(t)<\infty)$. 
Indeed, note that the root is born immediately in the process, hence $|G_0\cap \CD(t)|=1$. Clearly we have then $\phi_0(t)=\Pv( |G_0\cap \CD(t)| = 0).$
We prove by induction that $\phi_{k}(t)=\Pv(|G_k\cap \CD(t)| = 0)$. Indeed, since $\log(0)=-\infty$ and an empty integral equals $0$, we have
\[ \ba \phi_1(t)&= \Ev \left[ \exp\left\{ \int_{0}^t \log(0) \xi(\mathrm dx)  \right\}\right]=\Ev[0\ind_{\xi(t)>0} + 1\ind_{\xi(t)=0}] \\
&=\Pv(\xi(t)=0) = \Pv(|G_1\cap \CD(t)|= 0). \ea\]
Similarly,
\[ \ba \phi_k(t)&= \Ev \left[ \exp\left\{ \int_{0}^t \log(\Pv(|G_{k-1}\cap \CD^{(x)}(t-x)|=0)) \xi(\mathrm dx)  \right\}\right]\\ 
&=\Ev\left[\prod_{i: \sigma_i\le t} \left( \Pv(|G_{k-1}\cap \CD^{(i)}(t-\sigma_i)|=0) \right)\right]  \\
&=\Pv(|G_k\cap \CD(t)|=0). \ea\]
Clearly we have $\{|G_k\cap \CD(t)|=0\}\subseteq \{|G_n\cap \CD(t)|=0\}$ for all $n\ge k$, hence, it is easy to see that for any fixed $t\ge0$, $\phi_k(t)$ is non-decreasing in $k$. Since also $0<\phi_k(t)<1$, the pointwise limit
\be\label{eq::phi-1} \phi(t):=\lim_{k\to \infty} \phi_k(t)=\lim_{k\to \infty} \Pv( |G_k\cap \CD(t)|=0)\ee
exists, and also 
\be\label{eq::phi-2} \phi(t)= \lim_{k\to \infty} \phi_k(t) = \lim_{k\to \infty} \big(T_\xi \phi_k\big)(t) = \big(T_\xi \phi \big)(t).\ee
hence $\phi(t)=\lim_{k\to \infty} \phi_k(t)$ satisfies the fixed point equation \eqref{eq::fixpoint}. It is also 
 the \emph{smallest solution} of this fixed point equation in $\Omega_{[0,1]}$. Indeed, for any other fixed point function $f\in \Omega_{[0,1]}$, $\phi_0 \le f$ holds trivially. Apply the operator $T_\xi$ $k$ times to both sides to obtain $\phi_k\le f$, (recall that $T_\xi$ preserves ordering), and take the limit to obtain $\phi\le f.$

Finally, it is easy to see that $f(t)\equiv 1$ always satisfies \eqref{eq::fixpoint}. Indeed, since $\log 1=0$, the integral gives $\exp\{0\}=1$ for all $t\in \R^+$.
Hence, the process is explosive if and only if there is a solution $\phi(t)$ to \eqref{eq::fixpoint} with $\phi(t)<1$ for some $t\in \R^+$. 
\end{proof}

\begin{claim}\label{cl::noninc}The function $\phi(t)=\Pv(D(t)<\infty)$ is non-increasing in $t$. \end{claim}
\begin{proof}
We use the sequence of functions $(\phi_k)_{k\ge 0}$ as in the proof of Lemma \ref{lem::fixpoint}. Trivially, $\phi_0(t)$ is non-increasing, and since $T_\xi: \Omega_{[0,1]} \to \Omega_{[0,1]}$,   $\phi_k(t)\le 1$ for all $k\in \N, t\in \R^+$. We write 
\[ \phi_{k+1}(t)-\phi_{k+1}(s) = \Ev\left[ \exp\left\{ \int_{0}^s \log(\phi_{k}(t-x)) \xi(\mathrm dx)  \right\} \cdot\left( \exp\left\{ \int_{s}^t  \log(\phi_{k}(t-x)) \xi(\mathrm dx) \right\} -1  \right) \right]. \]
Note that the factor $\exp\left\{ \int_{s}^t  \log(\phi_{k}(t-x)) \xi(\mathrm dx) \right\} -1$ on the right hand side is nonpositive (since $\phi_{k}(t-x)<1$), hence the right hand side is at most $0$. To finish, note that $\phi(t)=\lim_{k\to \infty} \phi_k(t) \ge \lim_{k\to \infty} \phi_k(s)=\phi(s)$.  
 \end{proof}
A bit more sophisticated statement is to see the following:
\begin{claim}[Explosion can happen arbitrarily fast]\label{cl::fastexp} Let $\phi$ be the smallest solution to \eqref{eq::fixpoint} in $\Omega_{[0,1]}$.
Then either $\phi(t)\equiv 1$ or  $\phi(t)<1$ for all $t>0$. 
\end{claim}
\begin{proof} Suppose $\phi(t)\equiv 1$ in an interval $[0,t_0]$ but $\phi(t)<1$ for $t>t_0$. 
 Consider the function $\psi(t):=\phi(t+t_0)$ for $t>t_0$. Then $\psi(t)<1$ for all $t>0$, and further,
\be\label{eq::psi-1}\ba (T_\xi \psi)(t)&= \Ev\left[ \exp \left\{ \int_{0}^{t} \log(\phi(t+t_0-x)) \xi(\mathrm dx) \right\} \right] \\
&=\Ev\left[ \exp \left\{ \int_{0}^{t+t_0} \log(\phi(t+t_0-x)) \xi(\mathrm dx) \right\} \right], \\
 \ea\ee
since $\phi(t+t_0-x)=1$ for $x\in [t,t+t_0]$ results in $\int_{t}^{t+t_0}  \log(\phi(t+t_0-x)) \xi(\mathrm dx)=0$.
Note that the rhs of \eqref{eq::psi-1} equals $\phi(t+t_0)=\psi(t)$. As a result, $\psi(t)=(T_\xi \psi) (t)$ is also satisfied. To finish, recall that $\phi(t)$ is per definition the smallest solution to the fixpoint equation \eqref{eq::fixpoint}. This is a contradiction, since $\psi(t)$ is also a solution and $\psi(t)<\phi(t)$ on $[0,t_0]$.
\end{proof}
As usual in the theory of branching processes, we say that the BP \emph{survives} if $N(t)$, the size of the coming generation never reaches zero, that is, the event
$\{N(t)\ge 1 \ \forall t\ge 0\}$ holds. Similarly, we say that \emph{extinction occurs} or the process \emph{dies out} if $\{\exists t\ge 0, N(t)=0\}$, that is, the total size of the existing population $D(t)$ eventually stops increasing. 
\begin{claim}[No conservative survival possible]\label{cl::nocons}
Let $\xi$ be the reproduction process of an explosive BP, that is, let us assume that the equation $\phi(t)=\big(T_\xi \phi\big)(t)$ has a non-trivial solution $\phi(t)<1$ for $t>0$. Then, almost surely, the process explodes on survival, that is,
\[  \Pv(\{V<\infty \}\circ \{\text{the BP survives}\})=0, \]
where $A\circ B=(A\setminus B) \cup (B\setminus A)$ is the symmetric difference of the events $A,B$.
\end{claim}
\begin{proof} First, note that $\Pv(V=\infty)=\lim_{t\to \infty}\phi(t):=\phi(\infty)$. Per definition, $\phi(\infty)$ satisfies the equation
\[ \phi(\infty) = \big(T_\xi \phi\big)(\infty)= \Ev\left[ \prod_{i:\sigma_i<\infty} \phi(\infty)\right] = \Ev\left[ \phi(\infty)^{|\xi(\infty)|}\right].\]
Note that this is the exact same recursion as the one that the extinction probability of the BP satisfies. Indeed, the process goes extinct if and only if all the subtrees of the children of the root go extinct. Hence, $\Pv(\text{extinction occurs})=\Pv(V=\infty)=\Pv(\text{the process does not explode}).$ Note also that 
$\{\text{extinction occurs}\}\subseteq\{V=\infty\}$, since the first event, $\{N(t)=0\}$ implies that $\{D(t)<\infty\}$ by the last statement of Lemma \ref{lemma::rays}. Taking the complement of the events finishes the proof.
\end{proof}
\subsection{General methods to test explosivity}
The operator $T_\xi$ provides a method to show that a particular process is explosive. This is the content of the next lemma, that first appeared in \cite{Grey74} for the special case Example \ref{ex::agedep}. 
\begin{claim}[Test-functions]\label{cl::testfunction}
The branching process with reproduction function $\xi$ is explosive if and only if there exists $t_0>0$ and a function $f:[0,t_0] \to [0,1]$ with $f \not\equiv 1$ and 
\be\label{eq::test} f(t) \ge \big(T_\xi f\big) (t)\ee
for all $t\in [0,t_0]$.  
\end{claim}
\begin{remark}\label{rem::qu-xi}\normalfont Sometimes it will be easier to analyse the following operator:
\be\label{eq::qu-xi} \big(Q_\xi f\big)(t):= 1-\big(T_\xi (1-f)\big) (t)\ee
Note that if $\phi(t)$ solves \eqref{eq::fixpoint} then $\eta(t):=1-\phi(t)$ solves $\eta(t)=\big(Q_\xi \eta\big)(t)$. As a result, Claim \ref{cl::testfunction} can be rephrased using the operator $Q_\xi$: The branching process with reproduction function $\xi$ is explosive if and only if there exists $t_0>0$ and a function $f:[0,t_0] \to [0,1]$ with $f \not\equiv 0$ and $f(t) \le \big(Q_\xi f\big)(t)$ for all $t\in [0, t_0]$.
\end{remark}
\begin{claim}[Operator domination implies stochastic domination of explosion times]\label{cl::opdom}
Let $\xi$ \\and $\xi'$ be two reproduction functions with corresponding operators $T_\xi$ and $T_{\xi'}$. Suppose there exists a $t_0>0$ such that for every $t \in [0,t_0]$, for every non-increasing $f: [0,t_0]\to  [0,1]$ 
\be\label{eq::opdom} \big(T_\xi f\big) (t) \ge \big(T_{\xi'} f\big) (t)\ee
holds, then the explosivity of $BP_\xi$ implies the explosivity of $BP_{\xi'}$ and further, the explosion time $V_{\xi'}$ of $BP_{\xi'}$ is stochastically dominated by the explosion time $V_{\xi}$ of $BP_\xi$, that is, $V_{\xi'}\  {\buildrel d \over \le}\ V_\xi$.  
\end{claim}
\begin{proof}[Proof of Claim \ref{cl::testfunction}]If the BP is explosive, then $\phi(t)=\Pv(D(t)<\infty)<1$ for $t>0$ and it satisfies $\phi(t)=T_\xi \phi(t)$, see \eqref{eq::phi-2} and \eqref{eq::fixpoint}, hence, necessity follows. The existence of such an $f$ is sufficient, since then define 
\[ \wit\phi_0(t):= f(t)\ind_{t\in [0, t_0]} + \ind_{t>t_0}. \]
and then set recursively $\wit\phi_k(t):=\big(T_\xi \wit\phi_{k-1}\big) (t)$. Then, $ \wit\phi_0(t)\ge \big(T_\xi \wit \phi_0\big) (t)=\wit \phi_1(t)$ for $t\in [0,t_0]$ by \eqref{eq::test} and trivially for $t>t_0$, since $1=\wit \phi_0(t)$ and $T_\xi$ maps $\Omega_{[0,1]}$ into intself. As a result, we see that 
$\wit \phi_0(t)\ge \wit \phi_1(t)$ for all $t\ge 0$. Apply $T_\xi$ on both sides to obtain $\wit \phi_k \ge \wit \phi_{k+1}$ for all $k$ and $t>0$, since $T_\xi$ preserves ordering (see below \eqref{eq::txi-def}). We obtained  $\wit\phi_k(t)\not \equiv 1$, a sequence of functions non-increasing in $k$ with values in $[0,1]$, so 
\[ \wit\phi(t):=\lim_{k\to \infty} \wit \phi_k(t)\] 
exists and is strictly less than $1$ for some $t>0$. Further, by bounded convergence it satisfies \eqref{eq::fixpoint}. Hence, there is a nontrivial fixed point function that solves \eqref{eq::fixpoint} and so the process is explosive.
\end{proof}
\begin{proof}[Proof of Claim \ref{cl::opdom}]
In this proof, every quantity gets a subscript $\xi$ or $\xi'$ to indicate to which BP the quantity belongs to.
Suppose the domination of the operators holds and $BP_\xi$ is explosive. Lemma \ref{lem::fixpoint} implies that $\phi_\xi(t):=\Pv(D_\xi(t) <\infty)\not \equiv 1$, a non-increasing function solves the fixed point equation \eqref{eq::fixpoint}. As a result,
\[ \phi_\xi(t) = \big(T_\xi \phi_\xi \big) (t) \ge \big(T_{\xi'} \phi_\xi \big) (t). \]
This means that $\phi_\xi(t)$ serves as a proper test function for the operator $T_{\xi'}$, and Claim \ref{cl::testfunction} implies that $BP_{\xi'}$ is also explosive. Following the \emph{proof} of Claim \ref{cl::testfunction}, setting $\wit \phi_0(t):=\phi_\xi(t)$ we obtain that $\phi_\xi(t) \ge \phi_{\xi'}(t)$, the smallest fixed point function of the operator $T_{\xi'}$. Recall that $1-\phi(t)$ is the distribution function of the explosion time, see \eqref{eq::dist-explosion}. So,
\[ \phi_\xi(t)= \Pv(D_\xi(t) <\infty) =\Pv(V_\xi > t ) \ge \phi_{\xi'}(t) = \Pv(V_{\xi'} > t ) \]
and the stochastic domination follows.
\end{proof}
\section{Analysis of general reproduction processes}\label{s::3}
In this section, we investigate  general reproduction processes. First we study what happens if the expected reproduction measure has a mass at $0$, as well as it is finite in some neighborhood of the origin.  Then, we develop a comparison of point processes that we call \emph{stochastic domination around the origin} as well as \emph{coupling around the origin}. We show that this domination implies operator domination in the sense of Claim \ref{cl::opdom}. This new notion of coupling enables us to study and compare the behavior of the classical examples with different reproduction functions below in Section \ref{s::classic}. 

\subsection{Mass of $\xi$ at zero.}
In what follows, we investigate what happens if $\xi$ has a positive mass at $0$, i.e., $\Pv(\xi(0)=0)\neq 1$. 
The analysis of this problem for age-dependent branching processes, (Example \ref{ex::agedep}) was carried out first by Sevast'anov \cite[Theorems 7, 8, 9]{Sev67}. Then, for discrete birth-times, it was re-discovered for branching random walks by Dekking and Host \cite{DeHo91}. A good summary is given in the introduction of \cite{AmiDev13}.
Here we provide statements for the general case, i.e., not just for age-dependent BPs but for general reproduction functions. 
The next theorem reduces the question to study reproduction functions that have no mass at $0$ but the expected number of points in any compact interval $[0, K]$ is infinite.  

\begin{theorem}\label{thm::mass-zero}
Consider a branching process with reproduction function $\xi$. Then, 
\begin{enumerate}
\item The process is explosive if $\Ev[|\xi(0)|]>1$, (including also the case $\Ev[|\xi(0)|]=\infty$). 
\item The process is conservative if $\Ev[|\xi(0)|]<1$ and there exists a finite $t_0>0$ such that 
\be\label{eq::finite-measure} \Ev\left[  |\xi(t_0)|\right]<\infty,\ee
i.e., $\xi$ has a finite intensity measure in some open neighborhood of the origin. 
\item If $\Ev[|\xi(0)|]=1$ and 
\begin{enumerate}
\item[(i)] $\Pv(|\xi(0)|=0)=0$, then the process is explosive. 
\item[(ii)]  $\Pv(|\xi(0)|=0)>0$, and $\Ev[|\xi((0,t_0])|]=0$ for some $t_0>0$, then the process is conservative. 
\item[(iii)]  $\Pv(|\xi(0)|=0)>0$, and there exists a finite $t_0>0$ such that  for all $0<t<t_0$,
\be\label{eq::finite-measure} 0<\Ev\left[  |\xi(t_0)|\right]<\infty,\ee
i.e., $\xi$ has a finite but positive intensity measure in some open neighborhood of the origin, then the question whether the process is explosive or conservative can be reduced to study a related process $\xi'$ with $\Ev[|\xi'(0)|]=0$ and $\Ev[|\xi'(t)|]=\infty$ for all $t>0$.
 \end{enumerate}
 \end{enumerate}
\end{theorem}
\begin{remark}\normalfont
The main message of this theorem is that only those processes might be explosive that either satisfy the criterions of part (c) (i) or (iii) or have $\Ev[|\xi(t)|]=\infty$ for all $t>0$.
\end{remark}
An immediate corollary is the following.
\begin{corollary} Branching processes with finite expected intensity measure can never be explosive. In particular, no Poissonian process can be explosive with locally finite intensity measure.
\end{corollary}
Inhomogeneous Poisson point processes as reproduction functions that do not have Malthusian parameter are always conservative if the intensity measure is locally finite: for instance, an inhomogeneous Poisson process with intensity measure $\mu(t)=\exp\{x^\alpha\}$ for any $\alpha>1$, or any function that grows faster than exponential to infinity, is conservative. On the other hand, if e.g. $\mu(t)=1/(1-x)$ on $[0,1)$ then the process explodes `sideways', i.e., already a single individual produces infinitely many offspring in finite time.

We further remark that condition \eqref{eq::finite-measure} is strictly weaker than having \emph{finite Malthusian parameter}, i.e., the existence of a $\la>0$ such that 
\be\label{eq::mal-def} \Ev\left[\int_0^\infty \e^{-\la x} \xi(\mathrm dx)\right]=1.\ee
Indeed, consider \[ \xi:=X_1 \delta_{0} + X_2\delta_{1} +\sum_{i=1}^{X_3} \delta_{\sigma_i}\] for a triple of (not necessarily independent) random variables $(X_1, X_2, X_3)$ with $\Ev[X_1]<1, \Ev[X_3]<\infty$ and $X_2<\infty$ almost surely, but $\Ev[X_2]=\infty$ and $\sigma_i$ i.i.d. from an arbitrary distribution with $F_\sigma(0)=0$. That is, an individual gives immediately birth to $X_1$ many children, she has an additional $X_2$ many children at time $1$, and the rest of her total number of children has finite expectation and zero mass at $0$.  In this case, since $\Ev[X_2]=\infty$, no Malthusian parameter exists, yet the process is conservative since it satisfies the criterions in part (b) of Theorem \ref{thm::mass-zero}. 

\begin{proof}[Proof of Theorem \ref{thm::mass-zero}]
\emph{Part (a).} When $\Ev[|\xi(0)|]>1$, consider the fixed point equation \eqref{eq::fixpoint} at $0$:
\be\label{eq::phi0} \phi(0) = \Ev\left[ \exp\{ \log(\phi(0)) |\xi(0)|\}\right] = \Ev\left[ \phi(0)^{|\xi(0)|}\right].\ee
Note that this equation describes the extinction probability of a usual Galton-Watson branching process with offspring distribution $|\xi(0)|$. Since $\Ev[|\xi(0)|]>1$, this BP  is supercritical, and hence equation \eqref{eq::phi0} has a root in $[0,1)$. As a result, $\phi(0)<1$ and hence the process is explosive. 
We remark that this means that in the original BP, the subprocess where we consider only those individuals that are born $0$ time after their mother, i.e., \emph{all} birth-times are $0$, is supercritical and hence it contains an infinite cluster with positive probability.

\emph{Part (b).}The conditions guarantee the existence of an $\ve>0$ with $\Ev[|\xi(\ve)|]<1$. So, we can modify $\xi$ so that 
 every individual born in $[0,\ve]$ is born at $0$ instead, i.e., 
 \[ \xi':=\delta_0 \cdot |\xi(\ve)| + \xi[\ve, \infty].\]
 Since every birth-time in the process $\xi'$ is shorter or equal to the birth-times in process $\xi$, we have $T_{\xi'} \le T_\xi$, by Claim \ref{cl::opdom} and 
so if $\xi'$ is conservative, then $\xi$ is also conservative.
 
 Consider now the sub-process of $\xi'$ with only zero birth-times. 
 This process is subcritical, since $\Ev[|\xi'(0)|]<1$. As a result, each cluster of zero-birth individuals is finite almost surely. This implies that in the BP with reproduction function $\xi'$, any ray to infinity must intersect infinitely many non-zero birth-times, and each of these have at least length $\ve$. Hence, the total length of any ray is infinite.
 
 Another, more analytic argument is the following. 
 Even though $T_{\xi'}$ itself might not be a contraction on $[0, \ve)$, its second and more iterates do contract. Indeed, for any $t\in [0,\ve)$,
 \[ \big(T_{\xi'} f\big)(t) - \big(T_{\xi'} g\big)(t)= \Ev\left[ f(0)^{|\xi(\ve)|}- g(0)^{|\xi(\ve)|}\right]\le \Ev[|\xi(\ve)|] |f(0)-g(0)|.\]
 Since $\Ev[|\xi(\ve)|]<1$, the further iterates of $T_{\xi'}$ contract. As a result, $T_{\xi'}$ has a unique fixpoint on $[0, \ve)$: the constant $1$ function. Then, by Claim \ref{cl::fastexp}, this is the only fixpoint on $[0, \infty)$ as well and hence $BP_{\xi'}$ is conservative. Claim \ref{cl::opdom} implies then that $BP_\xi$ is also conservative. 
 
%
\emph{Part (c).} \emph{Case (i)}: In this case, every individual has precisely one child that is born at the same time as her. Call this child the first child. Then, there is an infinite ray with $0$ total length, namely, $x=11111\dots$ has $\tau_x=0$. Hence the process explodes, see Lemma \ref{lemma::rays}.

\emph{Case (ii)}: In this case the argument used in Part (b) can be repeated: now, the clusters of zero-birth time individuals form critical Galton-Watson processes and as a result they are finite almost surely. Hence, there must be infinitely many non-zero birth-times on any infinite ray, and these have at least length $t_0$. Hence, the process is conservative. 

\emph{Case (iii)}:
Let us do the following iterative procedure: consider all individuals that are born at the same time as their mother, i.e., $\tau_x= \tau_{x_{|-1}}$ and collapse these individuals to their mother, that is, the new reproduction function of $x_{|-1}$ is the union of the reproduction processes of all children of her that are born at the same time as her. Do this iteratively until a non-zero birth time individual is reached. In other words, we ``collapse'' each cluster $C$ of zero birth-time individuals into a single individual $v_C$, and all the non-zero birth-time children of all individuals in $C$ will be allocated as children of $v_C$. This means that the new reproduction process of the root is given by
\[ \xi'([0,t]) = \bigcup_{x\in G} \ind_{\tau_x =0} \xi_x((0,t]). \] 
Note that $\sum_{x\in G} \ind_{\tau_x =0}$ is the total progeny of a critical branching process, and hence $\Ev[\sum_{x\in G} \ind_{\tau_x =0}]=\infty$. Since $\Ev[|\xi_x((0,t]|)>0$ and the summands are independent, by Wald's equality, for any $t>0$,
\[ \Ev[ |\xi'([0,t])| ] = \infty, \quad \Ev[ |\xi'(0)|]=0.\]
Hence, we reduced Part (c) (iii) to a reproduction process with zero expected mass at the origin and infinite expected mass elsewhere.
\end{proof}
In what follows we study reproduction processes that belong to the missing class in Theorem \ref{thm::mass-zero}. That is, from now on we assume that $\xi$ is such that for all $t>0$
\be\label{eq::infty-crit} \Ev[ |\xi([0,t])| ] = \infty, \quad \Ev[ |\xi(0)|]=0.\ee

\subsection{Couplings around $0$ and tail coupling.}
In what follows, we develop comparison techniques to be able to compare two reproduction processes $\xi$ and $\xi'$ in terms of explosion.
First we start by defining `partial' couplings of random variables. 
\begin{definition}[Stochastic domination around the origin]\label{def::0-dom}
For two nonnegative random variables $X,Y$ with distribution functions $F_X(t), F_Y(t)$ we write
 \[ X\ { \buildrel d,0 \over \le }\ Y\] and say that $X$ is stochastically dominated by $Y$ around the origin, if there exists a $t_0=t_0(F_X, F_Y)>0$ such that $F_X(t) \ge F_Y(t)$ for all $t\in [0, t_0]$. 
 
 Similarly, we say that $X$ is stochastically dominated by $Y$ around infinity, or $X$ is \emph{tail-dominated} by $Y$, and write
 \[X\ { \buildrel d,\infty \over \le }\ Y, \] if there exist a $K=K(F_X, F_Y)>0$ such that $F_X(t)\ge F_Y(t)$ for all $t>K$.
\end{definition}
\begin{lemma} (1) $X$ is stochastically dominated by $Y$ around the origin if and only if there exists a $t_0>0$ and a coupling $(\wit X, \wit Y)$ of $X$ and $Y$ such that $\min \{\wit X, t_0\} \le \min\{\wit Y, t_0\}$. 

(2) Further, $X$ is tail-dominated by $Y$ if and only if there exists a $K>0$ and a  coupling $(\wit X, \wit Y)$ of $X$ and $Y$ such that $\max\{\wit X, K\} \le \max\{\wit Y, K\}$. Let us call such a coupling \emph{tail coupling}.
\end{lemma}
\begin{proof} Part (1). Suppose there is such a coupling. Then for all $t<t_0$, 
\[ F_Y(t)=\Pv(\wit Y\le t) = \Pv(\wit X \le \wit Y \le t) \le \Pv(\wit X \le t)=F_X(t).\] For the other way round, let $U\sim U[0,1]$ be a uniform random variable and let us define the generalised inverse of a function as $F^{-1}(s):=\sup\{ t: F(t) \le s\}$. Then the coupling $\wit X, \wit Y= F_X^{-1}(U), F_Y^{-1}(U)$ achieves such a coupling. Indeed, since $F_X(t)\le F_Y(t)$ for all $t<t_0$, we have $F^{-1}_X(s)\le F^{-1}_Y(s)$ for all $s\in [0, F_X(t_0)]$. It is straightforward to check that the marginal distributions are what they should be. 

The proof of Part (2) is analogous and is left to the reader.
\end{proof}
There is another way to formulate coupling at the origin of two random variables. Consider $X,Y$ as unit masses at locations $X, Y$, and define two point processes $\xi_X, \xi_Y$, each of which has a unit mass at a single point, namely at $X$ and at $Y$, respectively. With this notation, $X\ { \buildrel d,0 \over \le }\ Y$ if and only if there exists a coupling  that satisfies 
\[ |\xi_X(t)| \ge |\xi_Y(t)| \]
for all $t\le t_0$.
We arrived to the crucial tool to compare reproduction processes.
\begin{definition}[Stochastic domination of point processes around the origin]\label{def::point-proc-dom}
Let $\xi$ and $\xi'$ be two point processes. We say that $\xi$ is stochastically dominated by $\xi'$ around the origin and write $\xi\ { \buildrel d,0 \over \le }\ \xi'$ if there exists a $t_0>0$ and a coupling of $\xi, \xi'$, denoted by $(\wit \xi, \wit \xi')$ such that
\be\label{eq::partial-coupling} |\wit\xi (t) |\ge |\wit\xi'(t) |\ee
holds for all $t\in[0, t_0]$ with probability $1$ under the coupling.
\end{definition}
The advantage of this definition is that it makes sense also for processes $\xi$ that have $\Ev[|\xi(t)|]=\infty$ for all $t>0$. 
\begin{theorem}\label{thm::comparison-general} Let $\xi, \xi'$ be two reproduction processes of two branching processes $BP_\xi$ and $BP_{\xi'}$, respectively. If $\xi'\ { \buildrel d,0 \over \le }\ \xi$ and $BP_{\xi}$ is explosive, then so is $BP_{\xi'}$. Further, the explosion time $V_{\xi'}$ of the process with $\xi'$ is stochastically dominated by the explosion time $V_\xi$ of the process with $\xi$, that is, $V_{\xi'}\  {\buildrel d \over \le }\ V_\xi$. If $BP_{\xi'}$ is conservative, then so is $BP_{\xi}$.
\end{theorem}
\begin{proof}
Suppose $\xi'\ { \buildrel d,0 \over \le }\ \xi$, that is, there is a $t_0>0$ and a coupling of the processes with $|\xi'(t)| \ge |\xi(t)|$ for all $t\in[0,t]$. Then, for any non-increasing function $f:[0,t_0]\to [0,1]$ and
for $x'<x$ we have $\log f(t-x') < \log f(t-x)\le0$.  By the assumption on stochastic domination around the origin, consider a $t<t_0$. Then under the coupling
\[ \int_0^t \log f(t-x) \xi(\mathrm dx) \ge \int_0^t \log f(t-x) \xi'(\mathrm dx) \]
holds with probability $1$.
Exponentiating and taking expectations on both sides yields that 
\[  \big(T_{\xi} \psi \big)(t) \ge \big(T_{\xi'} \psi \big)(t), \]
hence, by Claim \ref{cl::opdom}, the operator $T_\xi$ dominates $T_{\xi'}$, finishing the proof of the first two statements.
The last statement is obvious by contradiction: if $BP_\xi$ is conservative but $BP_{\xi'}$ would explode, then by the first statement, $\xi$ would also be explosive, which is a contradiction.
\end{proof}
\section{Classical examples}\label{s::classic}
In this section we study well-known classical examples in terms of explosiveness.  We use the following notation. The (marginal) distribution function of a random variable $Y$  is denoted by $F_Y(t)=\Pv(Y\le t)$ and its generating function by $h_Y(s):=\Ev[s^Y]$. In the following examples, we assume that $\xi(\infty):=X\ge 0$, a discrete random variable having generating function $h_X(s)=\Ev[s^X]$. The birth times $(\sigma_i)_{1\le i\le X}$ are i.i.d. from distribution $F_\sigma(x)=\Pv(\sigma\le x)$ and finally, a pair of random variables $(I,C)$ with joint distribution function $F_{[I, C]}(s,t)=\Pv(I\le s, C\le t)$. 
\begin{example}[Age-dependent branching processes]\label{ex::agedep}\normalfont
In a usual age-dependent branching process, the root has a random number of children, $X$, with i.i.d.\ birth times from distribution $\sigma$.
In this case, 
\be\label{eq::xi-agedep} \xi(t)=\sum_{i=1}^X \ind_{\sigma_i\le t} \ee
Using this, the operator $T_\xi f$ is easy to calculate by first conditioning on $X$ and using the i.i.d. nature of $\sigma_i$ 
\[ \big(T_\xi f\big)(t)=\Ev\left[\prod_{i=1}^X \left( f(t-\sigma_i) \ind_{\sigma_i\le t} + \ind_{\sigma_i>t}\right)\right]=\Ev\left[\left(\int_{0}^t f(t-x) F_{\sigma}(\mathrm dx) + 1-F_{\sigma}(t) \right)^X\right]\]
and the operator becomes
\be\label{eq::t-agedep} \big(T_\xi f\big)(t)=h_X\left( 1-\int_0^t 1-f(t-x) F_\sigma(\mathrm dx) \right).  \ee
\end{example}
\begin{example}[Epidemics with contagious period]\label{ex::cont}\normalfont
In a model of an epidemic, each individual tries to transmit the infection to $X$ many other individuals, at i.i.d.\ times from distribution $\sigma$, but she can transmit the infection only in a random interval $[0,C]$ after being infected. That is, in this case
\be\label{eq::xi-cont} \xi(t)=\sum_{i=1}^X \ind_{ \sigma_i \le t } \ind_{\sigma_i \in [0,C]}.\ee
Note that setting $C\equiv \infty$ yields an age-dependent BP as in \eqref{eq::xi-agedep}.
Again, conditioning on $C, X$, we can use the i.i.d. nature of $\sigma_i$ to calculate the operator $T_\xi f$:
\be\label{eq::t-cont} \ba \big(T_\xi f\big)(t)&=\Ev\left[ \prod_{i=1}^X \left(f(t-\sigma_i) \ind_{ \sigma_i \le \min(t,C)} + (1- \ind_{ \sigma_i \le \min(t,C)}) \right)\right]\\  
&= \Ev\left[ \ind_{t>C} \prod_{i=1}^X \Big(f(t-\sigma_i) \ind_{ \sigma_i \le C} + \ind_{ \sigma_i > C} \Big) + \ind_{t<C} \prod_{i=1}^X \Big(f(t-\sigma_i) \ind_{ \sigma_i \le t} + \ind_{ \sigma_i > t} \Big)\right]  \\
&= \int_0^t h_X\left( 1-\int_0^c 1-f(t-x) F_{\sigma}(\mathrm dx) \right) F_C(\mathrm dc)\\  
&\quad + h_X\left( 1-\int_0^t 1-f(t-x) F_{\sigma}(\mathrm dx) \right) (1-F_C(t))
\ea\ee
\end{example}
\begin{example}[Epidemics with incubation times]\label{ex::inc}\normalfont
In this model of an epidemic, each individual tries to transmit the infection to $X$ many other individuals, at i.i.d.\ times from distribution $\sigma$, but she can transmit the infection only in a random interval $[I,\infty]$ after being infected. That is, in this case
\be\label{eq::xi-incu} \xi(t)=\sum_{i=1}^X \ind_{ \sigma_i \le t } \ind_{\sigma_i \in [I,\infty]}.\ee
Note that setting $I\equiv 0$ yields an age-dependent BP as in \eqref{eq::xi-agedep}.
Again, conditioning on $I, X$, we can use the i.i.d. nature of $\sigma_i$ to calculate the operator $T_\xi f$:
\be\label{eq::t-incu} \ba \big(T_\xi f\big)(t)
&=\Ev\left[ \ind_{I<t} \prod_{i=1}^X \left(f(t-\sigma_i) \ind_{ I<\sigma_i < t} + (1- \ind_{ I<\sigma_i < t}) \right) 
+ \ind_{I>t}\right]\\
&= \int_0^t h_X\left( 1- \int_i^t 1-f(t-x) F_{\sigma}(\mathrm dx) \right) F_I(\mathrm di)+ 1-F_I(t).\ea \ee
\end{example}
\begin{example}[Epidemics with incubation times and  contagious periods]\label{ex::cont}\normalfont
In this model of an epidemic, each individual tries to transmit the infection to $X$ many other individuals, at i.i.d.\ times from distribution $\sigma$, but she can transmit the infection only in a (possibly empty or infinite) random interval $[I,C]$ after being infected. That is, in this case
\be\label{eq::incucont} \xi(t)=\sum_{i=1}^X \ind_{ \sigma_i \le t } \ind_{\sigma_i \in [I,C]}.\ee
Note that setting $I\equiv 0$ yields \eqref{eq::xi-cont}, setting $C\equiv \infty$ gives \eqref{eq::xi-incu} and setting $I\equiv 0, C\equiv \infty$ results in \eqref{eq::xi-agedep}. For the operator $T_\xi$ of this process, see \eqref{eq::t-incu-cont} below.
\end{example}
\begin{example}[Backward process of epidemic models]\label{ex::backward}\normalfont
The backward version of the epidemic models above traces how the infection could have reached an individual in the population backward in time, and it was first noted in \cite{BaRe13}. That is, an individual has contact with $X$ many other individuals, at i.i.d.\ times from distribution $\sigma$, but she is only infected by the $i$th contact if she is in the (possibly empty or infinite) random interval $[I_i,C_i]$ after being infected, where the intervals $[I_i,C_i]$ are independent and independent of $\sigma_i$. That is, in this case
\be\label{eq::incucont-back} \xi(t)=\sum_{i=1}^X \ind_{ \sigma_i \le t } \ind_{\sigma_i \in [I_i,C_i]}.\ee
Note that setting $I_i\equiv 0$ yields the backward version of \eqref{eq::xi-cont}, setting $C_i\equiv \infty$ gives the backward version of \eqref{eq::xi-incu} and setting $I_i\equiv 0, C_i\equiv \infty$ results in the backward version of \eqref{eq::xi-agedep}, but in this last case the usual and the backward version are the same. See \eqref{eq::t-incu-cont-b} below for the formula for $T_\xi$ in this case.
\end{example}
\subsection{Unified notation for classical examples} 
In the sequel, we shall analyse the classical examples in terms of their explosivity and the relation between them. To do so, we abbreviate each process by a quadruple: the generating function of the total progeny, $h_X$, the distribution of the lifetime, $F_\sigma$, the distribution of the incubation time $F_I$ and finally the distribution of the contagious period $F_C$. In case $I$ or $C$ are equal to a constant with probability $1$, then we abuse the notation and write the constant instead of the distribution function. Thus, a quadruple $(h, F_\sigma, F_I, F_C)$ stands for a process as in \eqref{eq::incucont} with $I$ and $C$ being independent. In case $I$ and $C$ are both nontrivial and not independent, we give the joint distribution of $I, C$ by merging the last two entries $F_I, F_C$ and write  $F_{[I, C]}$ instead.
 We add a superscript $b$ for noting when we refer to a backward process, otherwise we always mean the usual (forward) version.
That is, an age-dependent BP as in \eqref{eq::xi-agedep} will be shortly denoted by $(h_X, F_\sigma, 0, \infty)$. A process with only incubation times/contagious period will be denoted by $(h_X, F_\sigma, F_I, \infty)$ and $(h_X, F_\sigma, 0, F_C)$, respectively. Their backward version are denoted by $(h_X, F_\sigma, F_I, \infty)^b$ and $(h_X, F_\sigma, 0, F_C)^b$, respectively.

\subsection{Comparison theorems for the classical examples.}\label{s::comp-class}
In this section we prove comparison theorems for the above listed classical examples.
A version of Theorems \ref{thm::comp1}  and \ref{thm::comp-tail} below -- without the contagious period $C$ -- appeared first in \cite{Grey74} for age-dependent BPs.   His proofs were analytic: he analysed the behaviour of the operator \eqref{eq::t-agedep}. Here, whenever possible, we provide new probabilistic proofs using the notion of coupling around the origin. We only fail to do so in Theorem \ref{thm::comp-tail}, that is, there our proof is analogous to that of Grey \cite{Grey74}. We explain in Example \ref{ex::tail-counter} below why a probabilistic proof using tail-coupling would fail.

The first theorem tells us that decreasing the birth times and making the end of the contagious interval can only help explosion, as long as the incubation times are all $0$.
\begin{theorem}\label{thm::comp1}
 Let $BP_\xi$ and $BP_{\xi'}$ be two age-dependent branching processes or epidemic models with contagious periods as in \eqref{eq::xi-agedep} or as in \eqref{eq::xi-cont} with reproduction functions $\xi$ and $\xi'$.  Let us assume that the birth time distributions $\sigma,\sigma'$  and (possibly infinite) contagious periods $C, C'$ satisfy
\[ \sigma'\  {\buildrel d,0 \over \le}\ \sigma,  \quad C\  {\buildrel d,0 \over \le}\ C'. \]
If $BP_\xi=(h_X, F_\sigma, 0, F_C)$ is explosive, then so is  $BP_{\xi'}=(h_X, F_{\sigma'}, 0, F_{C'})$, and the explosion time $V_{\xi'}$ of $BP_{\xi'}$ is stochastically dominated by the explosion time $V_\xi$ of $BP_{\xi}$. The same is true for the backward versions of the processes.
\end{theorem}
\begin{proof} Set $t_0$ as the minimum of $t_0(F_\sigma, F_{\sigma'}), t_0(F_C,F_{C'})$ from Definition \ref{def::0-dom}. Use the same copy of $X$ for the two processes, and couple $\sigma_i, \sigma'_i$ and $C, C'$ so that $\min\{ \sigma_i', t_0\} \le\min\{ \sigma_i, t_0\}$ and $\min\{C, t_0\} \le\min\{ C', t_0\}$ holds for all $i$. 
Then we obtain a coupling with $\ind_{ \sigma_i' \le t } \ind_{\sigma_i' \in [0,C']} \ge  \ind_{ \sigma_i \le t } \ind_{\sigma_i \in [0,C]}$, which yields 
 a coupling between $\xi$ and $\xi'$ such that $|\xi'(t)| \ge |\xi(t)|$ holds almost surely under the coupling for all $t \in [0, t_0]$. By Theorem \ref{thm::comparison-general}, whenever $\xi$ is explosive, then so is $\xi'$, and whenever $\xi'$ is conservative, then so is $\xi$. The stochastic domination of the explosion times also follows.
 For the backward versions, the same coupling results in $\ind_{ \sigma_i' \le t } \ind_{\sigma_i' \in [0,C_i']} \ge  \ind_{ \sigma_i \le t } \ind_{\sigma_i \in [0,C_i]}$, and hence the statement follows.
\end{proof}
The next theorem says that decreasing the incubation time (and increasing the end of the contagious interval) can only help explosion. 
\begin{theorem}\label{thm::comp2} Let $BP_\xi$ and $BP_{\xi'}$ be two age dependent branching processes or epidemic models with (possibly zero) incubation times and (possibly infinite) contagious periods as in \eqref{eq::xi-agedep} or as in \eqref{eq::xi-incu} with reproduction functions $\xi$ and $\xi'$. Let us assume that the (possibly zero) $I, I'$ and (possibly infinite) contagious periods $C, C'$ satisfy
\[ I'\  {\buildrel d,0 \over \le}\ I,  \quad C\  {\buildrel d,0 \over \le}\ C'. \]
If $BP_{\xi}=(h_X, F_\sigma, F_I, F_C)$  is explosive, then so is  $BP_{\xi'}=(h_X, F_\sigma, F_{I'}, F_{C'})$, and the explosion time $V_{\xi'}$ of $BP_{\xi'}$ is stochastically dominated by the explosion time $V_\xi$ of $BP_\xi$. The same is true for the backward versions of the processes. 
\end{theorem}
\begin{proof} In this case, using the same copy of $X$ and copies of $\sigma_i$ for the two processes, for a small enough $t_0$, the indicator variables can be coupled so that $\ind_{ \sigma_i \le t } \ind_{\sigma_i \in [I',C']} \ge \ind_{ \sigma_i \le t } \ind_{\sigma_i \in [I,C]}$ and so $|\xi'(t)|\ge |\xi(t)|$ for all $t$ under the coupling. Theorem \ref{thm::comparison-general} finishes the proof. For the backward versions, the same coupling results in $\ind_{ \sigma_i \le t } \ind_{\sigma_i \in [I_i',C_i']} \ge  \ind_{ \sigma_i \le t } \ind_{\sigma_i \in [I_i,C_i]}$, and hence the statement follows.
\end{proof}
An immediate corollary of Theorems \ref{thm::comp1} and \ref{thm::comp2} is the following
\begin{corollary}\label{cor:droppingIC} If at least one of the branching process $(h, F_\sigma, F_{[I,C]})$ or $(h, F_\sigma, F_{[I,C]})^b$ is explosive for $(I, C)\not \equiv (0, \infty)$, then so is  the age-dependent BP $(h, F_\sigma, 0, \infty)$. 
\end{corollary}

\begin{remark}\normalfont
We would also like to compare two epidemic models with non-zero incubation times when $\sigma' \ {\buildrel d,0 \over \le }\ \sigma$. However, with the usual coupling, $\ind_{ \sigma_i' \le t } \ind_{\sigma_i' \in [I,\infty]} \not\ge \ind_{ \sigma_i \le t } \ind_{\sigma_i \in [I, \infty]}$ and as a result the two processes cannot necessarily be compared.
\end{remark}
However, assuming that either the density or the mass at any point is dominated in one process by the other, we obtain the following
\begin{theorem}\label{thm::comp4} Let $BP_\xi$ and $BP_{\xi'}$ be two age dependent branching processes or epidemic models with (possibly zero) incubation times and (possibly infinite) contagious periods as in \eqref{eq::xi-agedep} or as in \eqref{eq::xi-incu} with reproduction functions $\xi$ and $\xi'$. Let us assume that the birth-time distribution $\sigma, \sigma'$ satisfy for all $t\in[0, t_0]$
\[ F_{\sigma'}(\mathrm dt)\ge F_{\sigma}(\mathrm dt). \]
If $BP_{\xi}=(h_X, F_\sigma, F_{[I,C]})$  is explosive, then so is  $BP_{\xi'}=(h_X, F_{\sigma'}, F_{[I,C]})$, and the explosion time $V_{\xi'}$ of $BP_{\xi'}$ is stochastically dominated by the explosion time $V_\xi$ of $BP_\xi$. The same is true for the backward versions of the processes. 
\end{theorem}
Since the proof uses the specific form of the operators $T_\xi$ and $T_{\xi'}$, we postpone the proof after the proof of Theorem \ref{thm::forw-back} below. 
\begin{remark}\normalfont Note that the disadvantage of this theorem is that it only allows us to compare two processes if they both have densities and/or if they both put mass \emph{on the same set of points}.
\end{remark}

The next comparison theorem tells us that if the total number of offspring $X'$ has `heavier tails' then $X$, then the explosion of the process with the lighter tail implies the explosion of the process with heavier tails. 
Intuitively, heavier tail behaviour implies more points around the origin, so we would like to say that if $Y$ tail-dominates $X$ then the BP with total offspring $Y$ is more likely to explode.   This is unfortunately not always the case. Consider namely the pair $\Pv(Y=0)=\Pv(X=K)=p_0$ while $\Pv(Y=K+\ell)=\Pv(X=K+\ell)$ for all $\ell\ge 1$. In this case $Y$ tail-dominates $X$ (and also $X$ tail-dominates $Y$), and a simple calculation shows that $h_Y(s)\ge h_X(s)$ for all $s\in (0,1)$, hence $h_Y$ is not steeper at $1$. We see below in the proof of Theorem \ref{thm::comp-tail}, that when everything else is the same, then the process with steeper generating function at $1$ is more likely to explode. 

This example is of course very artificial: the tail of the two distributions are eventually the same. One might suspect that when the tail distributions differ enough, then tail-domination would yield steeper generating functions, but this is still not the case,  even if we require that $F_X(\ell)>F_Y(\ell)$ for infinitely many values of $\ell$. The following counterexample provides with such a construction:

\begin{example}[Tail-domination may not imply steeper generating functions]\label{ex::tail-counter} Let $X$ and $Y$ be two nonnegative integer-valued random variables with $X\, {\buildrel d, \infty \over \le}\, Y,$ $K$ as in Definition \ref{def::0-dom} be an odd number. Let $\Pv(X=K)=\Pv(Y=0)=p_0$ and let $F_X(\ell) > F_Y(\ell)$ when $\ell$ is even and $F_X(\ell) = F_Y(\ell)$ when $\ell$ is odd, for all $\ell>K$. Then for all $s$, $h_Y(s)>h_X(s)$, that is, the generating function of $Y$ is not steeper at $1$.\end{example}
%
\begin{proof} Let $K$ be as in the Definition \ref{def::0-dom}. Then, 
\[ \Ev[s^X]-\Ev[s^Y]= \Ev[(s^X- s^Y) \ind_{Y\le K}]+ \Ev[(s^X-s^Y)\ind_{K<  Y}]. \]
The first term might be negative, and is the smallest when $X\equiv K$ and $Y\equiv 0$ on $\{Y\le K\}$, exactly as described in the example. Then 
\[ |\Ev[(s^X-s^Y) \ind_{Y\le K}]| = (1-s^K)   \Pv(Y\le K)= K\delta   \Pv(Y\le K)(1+o(1)), \]
where $\delta:=1-s$ and $o(1)\to 0$ as $\delta\to 0$. 
We need a bound on the second term. Since $F_X(\ell)\ge F_Y(\ell)$ for all even $\ell>K$ but $F_X(\ell)=F_Y(\ell)$ for all odd $\ell$, 
elementary calculation shows that
\[ \Ev[(s^X-s^Y)\ind_{K< Y}] = \sum_{i=\lceil K/2 \rceil}s^{2i}(\Pv(X=2i)-\Pv(Y=2i)) (1-s)\le \Pv(X> K)\delta.\]
As a result,
\[h_X(s)-h_Y(s)\le  \delta \left(\Pv(X> K) - K \Pv(Y\le K)(1+o(1))\right) <0 \]
when $\delta$ is small enough and $ p_0>1/(K+1)$. 
\end{proof}
As a result of this counterexample, we are left with the following comparison theorem about the distribution of the number of offspring:

\begin{theorem}\label{thm::comp-tail}
Let $BP_\xi$ and $BP_{\xi'}$ be two age dependent branching processes or epidemic models with (possibly zero) incubation times and (possibly infinite) contagious periods as in \eqref{eq::xi-agedep} or as in \eqref{eq::xi-incu} with reproduction functions $\xi$ and $\xi'$. Let us assume that there exists an $s_0<1$ that the generating functions of the total progeny $X,X'$ satisfy for all $s\in (s_0, 1)$ that
\[ h_{X'}(s)\le  h_{X}(s). \]
If $BP_{\xi}=(h_X, F_\sigma, F_{[I,C]})$  is explosive, then so is $BP_{\xi'}=(h_{X'}, F_\sigma, F_{[I,C]})$, and the  The same is true for the backward versions of the processes. 
\end{theorem}
We again postpone the proof. 
The next theorem say that the backward process explodes faster than the forward process of an epidemic:
\begin{theorem}[Explosion of the forward process implies explosion of backward process]\label{thm::forw-back}\ \\Suppose $(h_X, F_\sigma, F_{[I,C]})$ is explosive. Then so is  $(h_X, F_\sigma, F_{[I,C]})^b$, and the explosion time $V^b$ of the backward process is stochastically dominated by the explosion time of the forward process $V$, i.e., $V^b \ {\buildrel d \over \le V}$. As special cases we obtain that  the explosion of $(h_X, F_\sigma, 0, F_C)$ implies the explosion of $(h_X, F_\sigma, 0, F_C)^b$ and the explosion of $(h_X, F_\sigma, F_I, \infty)$ implies the explosion of $(h_X, F_\sigma, F_I, \infty)^b$.
\end{theorem}
\begin{proof}[Proof of Theorem \ref{thm::forw-back}] We show the statement for the general case $[I, C]$. The proof of the special cases is similar/can be reduced from the general case. Let us denote the operator as in \eqref{eq::txi-def} for the usual process by $T_\xi$, and $T_b$ for the backward version. Our goal is to show that the conditions of Claim \ref{cl::opdom} are satisfied and $T_\xi$ dominates $T_b$. The key is Jensen's inequality.
Indeed, for a fixed $t>0$ let us introduce the notation $A(t):=\{(i,c): 0\le i\le t, i\le c\le t\}$ for the triangle in the $(i,c)$ plane with corners $(0,0), (0,t), (t,t)$, and $B(t):=\{(i,c): 0\le i \le t, c>t\}$ for the infinite rectangle with corners $(0,t), (t,t), (0,\infty), (t, \infty)$. Finally, let us denote the complement of $A(t)\cup B(t)$ in the first quadrant by $C(t)$.
With this notation, using \eqref{eq::incucont} and the fact that $h_X(1)=1$,
\be \label{eq::t-incu-cont}\ba \big(T_\xi f \big)(t)&= \iint_{(i,c)\in A(t)} h_X \big( \Ev[f(t-\sigma_i) \ind_{i<\sigma_i<c} +\ind_{\sigma_i>c\cup \sigma_i<i} ]\big) F_{[I,C]}(\mathrm di, \mathrm dc) \\
&+\iint_{(i,c)\in B(t)} h_X\big(\Ev[ f(t-\sigma_i) \ind_{i\le \sigma_i\le t} +\ind_{\sigma_i>t\cup \sigma_i<i} ]\big)  F_{[I,C]}(\mathrm di, \mathrm dc)\\
&+ \iint_{i,c \in C(t)} h_X(1) F_{[I,C]}(\mathrm di, \mathrm dc).\ea
\ee
Note that $h_X$ is a generating function and hence convex. As a result, exchanging the integration with $h_X$, we have
\be\label{eq::t-incu-cont-b} \ba\big(T_\xi f\big)(t) &\ge 
h_X \Bigg(\iint_{(i,c)\in A(t)}  \Ev[f(t-\sigma_i) \ind_{i\le \sigma_i\le c} +\ind_{\sigma_i>c\cup \sigma_i<i} ] F_{[I,C]}(\mathrm di, \mathrm dc)  \Bigg. \\
&\quad+\iint_{(i,c)\in B(t)} \Ev[ f(t-\sigma_i) \ind_{i \le \sigma_i<t} +\ind_{\sigma_i>t \cup \sigma_i<i} ]  F_{[I,C]}(\mathrm di, \mathrm dc)\\
&\quad+ \iint_{i,c \in C(t)} 1 F_{[I,C]}(\mathrm di, \mathrm dc) \Bigg)=\big(T_b f\big) (t).\ea
\ee
A direct application of Claim \ref{cl::opdom} finishes the proof. 
\end{proof}

Now we are ready to prove Theorems \ref{thm::comp4} and \ref{thm::comp-tail}.
\begin{proof}[Proof of Theorem \ref{thm::comp4}]
We use the same notation as in formulas \eqref{eq::t-incu-cont} and \eqref{eq::t-incu-cont-b}. 
Note that for every $i, c< t_0$, and every function $f: \R\to [0,1]$ we have that the arguments of the generating function $h_X(\cdot)$ are decreased when switching to the distribution $F_{\sigma'}$.
Indeed, since $F_{\sigma'}(\mathrm dt)\ge F_{\sigma}(\mathrm dt)$ and $1-f(t-x)\ge 0$, we have for instance
\[ \ba \Ev[f(t-\sigma_i) \ind_{i<\sigma_i<c} +\ind_{\sigma_i>c\cup \sigma_i<i} ]&=1- \int_i^c 1-f(t-x) F_\sigma (\mathrm dx)\\
&\ge1- \int_i^c 1-f(t-x) F_{\sigma'} (\mathrm dx)\\
&=\Ev[f(t-\sigma_i') \ind_{i<\sigma_i'<c} +\ind_{\sigma_i'>c\cup \sigma_i'<i} ]  \ea  \]
by the monotonicity of $h_X$, the first and second line of \eqref{eq::t-incu-cont} and \eqref{eq::t-incu-cont-b} are both decreased. 
Hence, $\big(T_\xi f\big)(t)\ge \big(T_{\xi'} f\big)(t)$ holds for all $t\in [0, t_0]$. The theorem follows by an application of Claim \ref{cl::opdom}.
\end{proof}

\begin{proof}[Proof of Theorem \ref{thm::comp-tail}]
The assumption that $h_{X'}(s)\le h_{X}(s)$ for all $s\in (s_0, 1]$ implies that  $\big(T_\xi f\big)(t) \ge \big(T_{\xi'} f \big)(t)$ for all small enough $t>0$. Indeed, in formulas \eqref{eq::t-incu-cont} and \eqref{eq::t-incu-cont-b} the arguments of the function $h_X(\cdot)$ remain the same at every occurence, we only have to modify $h_X$ to $h_{X'}$ on each location to obtain $T_{\xi'} f$. $T_{\xi'} f \le T_\xi f$ thus holds if we show that the arguments of $h_X$ tend to $1$ as $t\to 0$. This is immediate from the fact that the expectation of the indicators $\ind_{i<\sigma_i<c}$ and $i\le \sigma_i\le t$ tend to $0$ as $t \to 0$ (since $c<t$ holds as well), and that $f \in [0,1]$.
\end{proof}

\subsection{Effect of contagious periods and incubation times}\label{s::effect1}
We have already seen in Theorem \ref{thm::forw-back} that the explosion of a forward process with contagious periods and/or incubation times implies the explosion of the backward process and also the explosion of the simple age-dependent process. 

In what follows, we investigate the opposite direction: under what circumstances can an age-dependent BP be stopped by superimposing incubation times/contagious periods on it? 
First, we make use of Claim \ref{cl::testfunction} to answer the question negatively for BP-s with contagious periods: as long as the total progeny follows a power-law distribution with exponent $\alpha \in (0,1)$, no explosive age-dependent BP becomes conservative by superimposing contagious periods on it. We also show that explosion is quite robust in terms of the total offspring distribution as well: the explosion of a process with birth-times following distribution $F_\sigma$ cannot be stopped by changing the offspring distribution $X$ to have lighter tails, as long as its power-law exponent stays in the interval $[0,1)$. We mention that the first forms of these theorems were developed by Gulikers and the author and appeared in the Master thesis of Gulikers \cite{Gul14}. The proofs provided here are different.
First a definition:
\begin{definition}[Power-law tail behavior] We say that the random variable has regularly varying tail with power-law exponent $\al\in (0,1)$ if there exists a $K>0$ and a function $L(\cdot)$ that varies slowly at infinity\footnote{For the definition, see below Definition \ref{def::slowvar}.} such that for all $x\ge K$
\be\label{eq::power-law-alpha} \Pv\left(X \ge x \right)=L(x)/ x^\al.\ee
\end{definition}
Karamata's theorem \cite{Kar30} or \cite[Theorem 1.7.1]{BiGoTe87} tells us that in this case the generating function of $X$ satisfies: 
\begin{proposition} Let $X$ be a random variable with tail behavior as in \eqref{eq::power-law-alpha}. Then there exists an $s_0\in[0,1)$ such that the generating function of $X$ satisfies for all $s\in [s_0, 1)$
\be\label{eq::gen-alpha} h_X (s) = 1- (1-s)^\al L\big(\tfrac{1}{1-s}\big).\ee
\end{proposition}
\begin{theorem}[Contagious periods do not matter `at all']\label{thm::cont-nomatter2}
Let $(h_X, F_\sigma, 0, \infty)$ be an age dependent branching process with offspring distribution $X$ as in \eqref{eq::power-law-alpha} for some $\al\in (0,1)$ and let $F_{[I,C]}$ be the joint distribution function of the incubation time and the contagious period $I, C$, and $F_I$ the marginal distribution of $I$. Let us further require the existence of  $t_0, \delta>0$ such that the conditional distribution $\Pv(C>t | I=i)>\delta$ for all $t>i$ with $i, t\in[0, t_0]$.   
Then, the explosivity of the process $(h_X, F_\sigma, F_I, \infty)$ implies the explosivity of the processes $(h_X, F_\sigma, F_{[I,C]})$ and $(h_X, F_\sigma, F_{[I,C]})^b$. 
\end{theorem}
\begin{remark}\normalfont
Note that the conditions of Theorem \ref{thm::cont-nomatter2} are satisfied in the following natural cases: 

(1) $I$ and $C$ are independent random variables with $F_C(t_0)<1$ for some $t_0>0$, with $\delta:=1-F_C(t_0)$. In this case, explosion of the process $(h, F_\sigma, F_I, \infty)$ implies the explosion of $(h, F_\sigma, F_I \times F_C)$. 

(2) $C\ {\buildrel d \over =}\ I+Y$ for some nonnegative random variable $Y$ with $\delta=\Pv(Y > 0)>0$ with $I, Y$ independent. That is, when the starting point $I$ and the length of the interval $[I, C]$ are independent.

\end{remark}
An immediate corollary is the following:
\begin{corollary}[Contagious periods do not matter]\label{thm::cont-nomatter}
Let $(h_X, F_\sigma, 0, \infty)$ be an age dependent branching process with offspring distribution as in \eqref{eq::power-law-alpha} for some $\al\in (0,1)$ and let $F_C$ be the distribution of the contagious period $C$ so that $C\not \equiv 0$, that is, for some $t_0>0$ let $F_C(t_0)<1$. Then, the explosivity of the process $(h_X, F_\sigma, 0, \infty)$ implies the explosivity of the processes $(h_X, F_\sigma, 0, F_C)$ and $(h_X, F_\sigma, 0, F_C)^b$.
\end{corollary}
Below in Corollary \ref{cor::power-law-change} we show that the precise exponent $\al$ of the degree distribution also does not matter, as long as a \emph{lower bound} $\Pv(X\ge x)\ge c/x^\al$ hold for some $\al<1$. As a result, the statement of Corollary \ref{thm::cont-nomatter} can be strengthened to hold for \emph{plump distributions} as in Definition \ref{def::plump1}. 

\begin{proof}[Proof of Theorem \ref{thm::cont-nomatter2}] Let us denote the operator corresponding to $(h_X, F_\sigma, F_{[I,C]})$ by $T_{[I,C]}$ and the operator corresponding to $(h_X, F_\sigma, F_I, 0)$ by $T_I$, and define as in Remark \ref{rem::qu-xi}  $\big(Q_{[I,C]} f\big)(t):=1- (T_{[I,C]} (1-f))(t), \big(Q_I f\big)(t):=1- (T_I (1-f))(t)$. With notation as in the proof of Theorem \ref{thm::forw-back}, it is elementary to check using \eqref{eq::t-incu} and \eqref{eq::t-incu-cont} that for a non-decreasing function $f$ with values in $[0,1]$, 
\be\label{eq::qi-1} \big(Q_{I}f\big)(t)  = \int_0^t \Big( \int_i^t f(t-x) F_\sigma (\mathrm dx) \Big)^\al  L\Big(1/\int_i^t f(t-x) F_\sigma (\mathrm dx)\Big )F_{I} (\mathrm di)
\ee
and
\[ \ba \big(Q_{[I,C]}f\big)(t) & = \iint\limits_{(i,c)\in A(t)} \Big( \int_i^c f(t-x) F_\sigma (\mathrm dx) \Big)^\al L\Big(1/\int_i^c f(t-x) F_\sigma (\mathrm dx)\Big )F_{[I,C]} (\mathrm di, \mathrm dc)\\
& \quad+ \iint\limits_{(i,c)\in B(t)} \Big( \int_i^t f(t-x) F_\sigma (\mathrm dx) \Big)^\al L\Big(1/\int_i^t f(t-x) F_\sigma (\mathrm dx)\Big )F_{[I,C]} (\mathrm di, \mathrm dc)\\
& \ge  \int_0^t \Pv(C>t| I=i ) \Big( \int_i^t f(t-x) F_\sigma (\mathrm dx) \Big)^\al L\Big(1/\int_i^t f(t-x) F_\sigma (\mathrm dx)\Big ) F_I(\mathrm di),
\ea \]
where we obtained the last line by only considering a lower bound on $B(t)$, since there the integrand does not depend on the value of $c$.
Since we assumed that $\Pv(C>t | I=i)>\delta$ for all $i<t,  t\in[0, t_0]$, we obtain for $t\in [0,t_0]$ that
\[ \big(Q_{[I,C]}f\big)(t) \ge \delta \big(Q_{I}f\big)(t).\]
Let us now assume that the process $(h_X, F_\sigma, F_I, \infty)$ explodes. Lemma \ref{lem::fixpoint} implies that $\phi(t)=\Pv(D_I(t)< \infty)\not \equiv 1$ solves
$\phi(t)=T_I \phi (t)$ and hence $\eta(t):=1-\phi(t) \not \equiv 0$ is a non-decreasing function that solves $\eta(t) =(Q_I \eta)(t)$.
Note that if the slowly varying function in \eqref{eq::qi-1} would not be present, then we would have $\Big(Q_I b\eta\Big)= b^\al \eta(t)$ for any constant $b>0$. 
We thus would like a lower bound on $\Big(Q_I b\eta\Big)$. For this we note that the argument of the slowly varying function is the smallest if $i=0$.
 So let us set for a fixed $c>0$, 
 \[ t_1(b):=\max\left\{ s: \frac{L\left(1/ (b\int_0^s \eta(t-x)  F_\sigma(\mathrm dx))\right)}{L\left(1/ \int_0^t \eta(t-x)  F_\sigma(\mathrm dx)\right)} \ge 1/2 \right\}, \]
 that is, by the fact that $\eta$ is non-decreasing, we have that for all $t<t_1(b)$, the ratio of the slowly varying functions
 \[ \frac{L\left(1/ (b\int_i^t \eta(t-x)  F_\sigma(\mathrm dx))\right)}{L\left(1/ \int_i^t \eta(t-x)  F_\sigma(\mathrm dx)\right)} \ge 1/2 \]
 for all $i, t<t_1(b)$. Then set $t_2(b):=\min\{t_0, t_1(b)\}.$
We make use Remark \ref{rem::qu-xi} and show that for a constant $b$, $b \eta(t)\not\equiv0$ satisfies $b \eta(t)\le \big(Q_C b\eta\big)(t)$ on $[0, t_2(b)]$ and hence the process is explosive for any $C\not\equiv 1$. 
Indeed, for $t\in [0, t_2(b)]$,
\[ \big(Q_{[I,C]} b\eta\big)(t) \ge  \delta\big(Q_I b\eta\big)(t)  \ge  \frac\delta2 b^\al \eta(t).\]
The right hand side is at most $b \eta(t)$ whenever $\frac12 \delta b^{\al-1}\ge 1$. This can be satisfied by choosing $b$ small enough. Hence, the forward process explodes by Claim \ref{cl::testfunction}. The explosion of the backward process $(h_X, F_\sigma, F_{[I,C]})^b$ then follows by Theorem \ref{thm::forw-back}.
\end{proof}

The next theorem tells us that an explosive age-dependent BP can be stopped from explosion by superimposing a conservative incubation time on it.
\begin{theorem}[Incubation times do matter]\label{thm::incu-matter}
Let $(h_X, F_\sigma, F_I, \infty)$ be an epidemic model with incubation times. If at least one of the age-dependent processes $(h_X, F_\sigma, 0, \infty), (h_X, F_I, 0, \infty)$ is conservative, then so is $(h_X, F_\sigma, F_I, \infty)$ and $(h_X, F_\sigma, F_I, \infty)^b$. In other words, the explosivity of both $(h_X, F_\sigma, 0, \infty)$ and $(h_X, F_I, 0, \infty)$ is necessary for the explosivity of $(h_X, F_\sigma, F_I, \infty)$.
\end{theorem} 
\begin{proof} We show that the backward process $(h_X, F_\sigma, F_I, \infty)^b$ is conservative if any of the processes  $(h_X, F_\sigma, 0, \infty)$ or $(h_X, F_I, 0, \infty)$ is conservative. Then, Theorem \ref{thm::forw-back} shows that $(h_X, F_\sigma, F_I, \infty)$ is also conservative.
Indeed, consider the three processes 
\[  \xi(t)=\sum_{i=1}^X \ind_{ \sigma_i \le t } \ind_{\sigma_i \in [I_i,\infty]},\quad   \xi'(t)=\sum_{i=1}^X \ind_{ \sigma_i \le t }, \quad  \xi''(t)=\sum_{i=1}^X \ind_{ I_i \le t }\]
and note that both $|\xi(t)|\le |\xi'(t)|$ and $|\xi(t)|\le |\xi''(t)|$ holds for all (not just small enough) $t>0$. As a result, $\xi$ is stochastically dominated by both $\xi'$ and $\xi''$ around the origin. Theorem \ref{thm::comparison-general} finishes the proof.
\end{proof}
\section{Min-summability and its consequences}\label{s:minsum}
A powerful tool to analyse the explosion of age-dependent branching processes is the so-called \emph{min-summability}, a criterion developed by Amini \emph{et al} \cite{AmiDev13}. Here we adapt the needed definitions and the main theorem from \cite{AmiDev13} to our notations. 
\begin{definition}Let $(a_k)_{k\in \N}$ be a sequence of real numbers with $a_k\to \infty$. We say that the distribution $F_\sigma$ is $a_k$-summable if
\[\sum_{k=1}^\infty \min\{\sigma_{k,1}, \sigma_{k,2}, \dots \sigma_{k, a_k}  \} <\infty,\]
where $(\sigma_{k,j})_{k,j\in \N}$ are i.i.d. copies of $\sigma$.
\end{definition}
\begin{definition}[Plump distributions]\label{def::plump1}
We say that the distribution of $X$ is plump if there exist positive constants $c, \de, x_0>0$ such that for all $x>x_0$, 
\be\label{eq::plump} \Pv(X>x)\ge \frac{c}{x^{1-\de}}. \ee
\end{definition}
A more restrictive definition is to require from the tail of the distribution to stay between the power-law regime:
\begin{definition}[Plump power-laws]\label{def::plump2}
We say that the distribution of $X$ is plump power-law if there exist positive constants $c, \de, x_0>0$ such that for all $x>x_0$, 
\be\label{eq::plump2} \frac{c}{x^{1-\de}}\le \Pv(X>x)\le \frac{c}{x^{\de}}. \ee
\end{definition}
\begin{remark}\normalfont Without loss of generality, (by possibly modifying $\delta, x_0$) we can assume that $c=1$ in Definitions \ref{def::plump1} and \ref{def::plump2}.
\end{remark}
Note that the notion of plump distributions is much weaker than having a power-law tail behavior. In fact, any distribution that satisfies \eqref{eq::power-law-alpha} for $\alpha\in(0,1)$ is plump, but so are distributions with logarithmic tails, e.g. $\Pv(X>x)=c/\log x$ is a plump distribution. The plump power-law definition allows for distributions where the power-law exponent might vary infinitely many times between different values of $\alpha$ but it only has fatter tails than power-laws finitely many times. The next notion is min-summability of branching process trees: the idea is, that the sum of the minimal birth-time in each generation should be finite. More precisely,
\begin{definition}[Min-summability of random trees]
Let $T$ be a rooted infinite random tree, and let $Z_k$ denote the number of vertices at graph distance $k$ from the root. 
We say that the random tree with  i.i.d. $\sigma$-distributed edge-weights is min-summable if 
\[ \Pv\left( \sum_{i=1}^\infty  \min\{ \sigma_{k,1}, \dots, \sigma_{k, Z_k}\} < \infty \right) =1.\]
\end{definition}
It is not hard to show (see e.g. \cite[Claim 0.1]{AmiDev13}) that min-summability is a tail event and hence it happens with probability $0$ or $1$ conditioned on the survival of the BP.
Here we rephrase the theorem about the equivalence of min-summability and explosion.
\begin{theorem}[Min-summability\,$=$\,explosion,\cite{AmiDev13}]\label{thm::amidev1} Let $\mathcal W_{M}(X)$ denote the set of weight distributions that are min-summable for a Galton-Watson BP conditioned on survival with offspring distribution $X$, and let $\mathcal W_{E}(X)$ denote the set of weight distributions $F_\sigma$ such that the age-dependent BP $(h_X, F_\sigma, 0, \infty)$ is an explosive process. Then, if $X$ is a plump distribution then $\mathcal W_{M}(X)=W_{E}(X)$.
\end{theorem}
In the sequel, we write 
\be\label{def::inverse} F^{(-1)}(y):=\inf\{t\in \R: F(t)\ge y \}\ee the generalised inverse of the distribution function $F$.

\begin{theorem}[Min-summability criterion, \cite{AmiDev13}]\label{thm::amidev2}
Given a plump offspring distribution $X$, let $x_0 > 1$ be large
enough such that the condition \eqref{eq::plump} holds for all $x \ge x_0$. Define the function
$h:\N \to \R^+$ as follows:
$h(0) = x_0$ and $h(n +1) = F^{(-1)}_X(1-1/h(n))$  for all $n \ge 1$.
Then for a weight distribution $F_\sigma$, the branching process $(h_X, F_\sigma, 0, \infty)$ is min-summable and hence also explosive  if and only if
\be\label{eq::explosive-crit-1} \sum_{k=1}^\infty F_\sigma^{(-1)}(1/h(n))<\infty.\ee
\end{theorem}
This theorem is extremely powerful and has many important implications: it shows that explosion is a fairly robust property. This is what we investigate below.
\subsection{Explosion is a robust property}\label{s::robust}
In this section, we show that the explosion of the classical examples are quite robust, but before that, we provide an equivalent criterion to \eqref{eq::explosive-crit-1} for plump power-laws that reveals the robustness of this theorem better. 
After that, we show that explosion is closed under such operations as 1) changing the degree distribution to a different power-law, 2) taking sums, maximum, binomial thinning of birth-time distributions that explode.
\begin{lemma}\label{lemma::integral} For a plump power-law distribution $X$, and a birth-time distribution $F_\sigma$, the age-dependent BP $(h_X, F_\sigma, 0, \infty)$ is explosive if and only if for a small enough $\ve>0$, and some arbitrary constant $C>0$, 
\be\label{eq::explosive-crit-2}
\int_0^\ve F^{(-1)}_\sigma\left( \frac{1}{\exp\{ C/y \}}\right) \frac1y \mathrm dy <\infty.
\ee
Equivalently, if and only if
\be\label{eq::explosive-crit-3}
\int_{1/\ve}^\infty F^{(-1)}_\sigma\left( \e^{- Cu}\right) \frac1u \mathrm du<\infty.
\ee
\end{lemma}
\begin{remark}\normalfont Note that the value of $C$ can be scaled out by changing variables. 
\end{remark}
\begin{proof} We show that \eqref{eq::explosive-crit-2} is equivalent to \eqref{eq::explosive-crit-1} for plump power-law distributions. Without loss of generality we will assume that $c=1$ in the definition of plump power-law distributions. First, let us sandwich the distribution of $X$ between $X'$ and $X''$ so that $\Pv(X'\ge x)=1/x^{1-\de}$ and $\Pv(X''\ge x)=1/x^{-\de}$ hold for all $x>x_0$. 
Let $h'(n), h''(n)$ be defined by the same recursion as $h(n)$, for the distribution $X'$ and $X''$ instead, respectively. Then, $h'(n) \le h(n)\le h''(n)$  holds for all $n$, hence \[ F^{(-1)}_\sigma(1/h''(n))\le F^{(-1)}_\sigma(1/h(n))\le F^{(-1)}_\sigma(1/h'(n))\] and so summability of the largest implies summability of the rest and divergence of the smallest implies divergence of the other two. 

First we show that convergence of the integral implies explosion of the process $(h_X, F_\sigma, 0, \infty)$.
Let $\alpha:=1-\de$, in the definition of plump distribution $X$. By the defining recursion of $h'(n)$, we obtain $h'(n)=(x_0)^{1/\alpha^n}$ Then, we write
\be\label{eq::bound11} \int_0^\ve F^{(-1)}_\sigma\left( \frac{1}{\exp\{ C/y \}}\right) \frac1y \mathrm dy=\sum_{n=k}^\infty \int_{\alpha^{n+1}}^{\alpha^{n}} F^{(-1)}_\sigma\left( \frac{1}{\exp\{ C/y \}}\right) \frac1y \mathrm dy,\ee
where $k$ can be either  chosen as $k:=\max \{n: \al^n >\ve\}$ or $k:=\min \{n: \al^n \le \ve\}$ depending on which direction of the if-and-only-if we look at.
Note that 
\be\label{eq::bound12}  F^{(-1)}_\sigma\left( \frac{1}{\e^{C/\alpha^{n+1}}}\right) (1-\alpha) \le \int_{\alpha^{n+1}}^{\alpha^{n}} F^{(-1)}_\sigma\left( \frac{1}{\e^{C/y}}\right) \frac1y \mathrm dy \le F^{(-1)}_\sigma\left( \frac{1}{\e^{C/\alpha^n}}\right) \frac{1-\alpha}{\alpha}, \ee
hence, the integral in \eqref{eq::explosive-crit-2} converges if and only if 
\be\label{eq::sum-3} \sum_{n=k}^{\infty} F^{(-1)}_\sigma\left( \frac{1}{\e^{C/\alpha^n}}\right) = \sum_{n=k}^{\infty} F^{(-1)}_\sigma\left( \frac{1}{h'(n)}\right) < \infty,\ee
where we have set $\e^{C}:=x_0$. Combination with Theorem \ref{thm::amidev2} finishes the proof.
For the other direction, note that the previous argument can be repeated with $\alpha:=\delta$ as well, and if the integral diverges, then so does the sum on the rhs of \eqref{eq::sum-3} with $h''(n)=(x_0)^{1/\alpha^n}$. Conservativeness of the process $(h_{X''}, F_\sigma, 0, \infty)$ and also that of $(h_{X}, F_\sigma, 0, \infty)$ then follows by Theorem \ref{thm::amidev2} again.
\end{proof}
\begin{corollary}[Power-law exponents can be changed]\label{cor::power-law-change}
Let $X$ be a plump power-law distribution and $F_\sigma$ is a weight-distribution. If $(h_X, F_\sigma, 0, \infty)$ is explosive, then $(h_Y, F_\sigma, 0, \infty)$ is explosive for any other plump power-law distribution $Y$. 
In particular, let $h_\al$ be the generating function of a random variable $X$ that has a power-law tail behaviour with exponent $\al\in(0,1)$, as in \eqref{eq::power-law-alpha}. Then, if $(h_\al, F_\sigma, 0, \infty)$ is explosive for some $\al \in (0,1)$, then it is explosive for all $\al\in (0,1)$. The same holds for $(h_X, F_\sigma, 0, F_C)$ and $(h_Y, F_\sigma, 0, F_C)$ for arbitrary contagious period $C$ with $\Pv(C>0)>0$.
\end{corollary}
\begin{proof} Let $X$ be a power-law distribution plump distribution lower bound exponent $\al=1-\de$ while $Y$ be a power-law exponent with lower bound exponent $\al'\in(0,1)$. According to Lemma \ref{lemma::integral} if $(h_X, F_\sigma, 0, \infty)$ is explosive, then for some $C$ the integral in \eqref{eq::explosive-crit-2} converges. Since $Y$ is also a plump power-law, and the criterion is independent of the value of $\al$, for $\al'$ the same criterion applies and hence the process  $(h_{Y}, F_\sigma, 0, \infty)$ also explodes. The second statement follows by setting $X$ and $Y$ follow strict power laws as in \eqref{eq::power-law-alpha}. The last statement is a simple combination with Corollary \ref{thm::cont-nomatter}.
\end{proof}

Next we harvest some simple consequences of the min-summability criterion.
\begin{theorem}[Max, sum and thinning of birth-times still explodes]\label{thm::max-birthtime}
Let $X$ be a plump power-law distribution, $\sigma, \gamma$ be two (independent) birth-time distributions. Further, let $\ind_p$ be a Bernoulli random variable with mean $p\in (0,1]$. Then, 
if $(h_X, F_\sigma, 0, \infty)$ and $(h_X, F_\gamma, 0, \infty)$ are both explosive, then the following operations on the birth-time distribution produce an explosive BP:
\begin{enumerate}
\item Multiplying by a nonnegative constant $C\ge0$: $(h_X, F_{C\sigma}, 0, \infty)$ is explosive.
\item Binomial thinning of each individual with probability $p\in(0,1)$: $(h_X, F_{\ind \sigma+ (1-\ind)\infty}, 0, \infty)$ is explosive.
\item Taking the maximum or minimum of two random variables:  $(h_X, F_{\max\{\sigma, \gamma\}}, 0, \infty)$ and  $(h_X, F_{\min\{\sigma, \gamma\}}, 0, \infty)$ is explosive.
\item Taking the sum of two independent random variables: $(h_X, F_{\sigma+\gamma}, 0, \infty)$ is explosive.
\end{enumerate}
\end{theorem}
\begin{remark}\normalfont The theorem remains valid also when we only assume that $X$ is a plump distribution. In this case, the proof is analogous to the one below, but one has to work with the sum \eqref{eq::explosive-crit-1} directly.
\end{remark}
By repeated use of the theorem, we obtain the following corollary.
\begin{corollary}Let $X$ be a plump power-law distribution and let $\sigma_1, \sigma_2, \dots, \sigma_k$ be independent but not necessarily identically distributed birth-time distributions and $\al_i\ge 0$ nonnegative numbers for $i\le k$. Then, if for all $1\le i\le k$ the processes $(h_X, F_{\sigma_i}, 0, \infty)$ are explosive then the processes $(h_X, F_{\max_{i\le k}\al_i\sigma_i}, 0, \infty)$ and $(h_X, F_{\sum_{i\le k} \al_i\sigma_i}, 0,\infty)$ also explode.
\end{corollary}
\begin{proof}[Proof of Theorem \ref{thm::max-birthtime}]
For part (a), note that $F_{C\sigma}(t)=F_\sigma(t/C)$ hence $F_{C\sigma}^{(-1)}(u)=C F_\sigma^{(-1)}(u)$. 
For part (b), $F_{\ind \sigma+ (1-\ind)\infty}(t)=\Pv(\ind \sigma<t)=p F_\sigma(t)$. Hence $F_{\ind \sigma+ (1-\ind)\infty}^{(-1)}(u)=F_\sigma^{(-1)}(u/p)$. 
By Lemma \ref{lemma::integral} if $F_\sigma^{(-1)}$ is integrable as in \eqref{eq::explosive-crit-2} then so are these two transforms of it. 

For part (c) and (d), we first show that both $F_{\max\{\sigma_1,\sigma_2\}}(t)=(F_\sigma(t))^2$ and $F_{\max\{\gamma_1,\gamma_2\}}=(F_\gamma(t))^2$ form an explosive process with offspring distribution $X$, where $\sigma_i, \gamma_i$ are i.i.d. copies of $\sigma$ and $\gamma$, respectively.
Indeed, $F_{\max\{\sigma_1, \sigma_2\}}^{(-1)}(y)=F_{\sigma}^{(-1)}(y^{1/2})$ and hence
\[ \int_0^\ve F_{\max\{\sigma_1, \sigma_2\}}^{(-1)}\left(\frac{1}{\e^{C/y}}\right)\frac1y \mathrm d y=  \int_0^\ve F_{\sigma}^{(-1)}\left(\frac{1}{\e^{C/(2y)}}\right) \frac1y\mathrm d y<\infty,\]
and the latter integral is finite by Lemma \ref{lemma::integral} since $(h_X, F_\sigma, 0, \infty)$ is explosive. The proof for $(F_\gamma(t))^2$ is analogous.
Then, note that $\Pv(\max\{\sigma, \gamma\}<t)= F_\sigma(t)  F_\gamma(t)$ and hence for every $t$,
\[  \min\{ (F_\sigma(t))^2, (F_\gamma(t))^2 \}\le F_\sigma(t)  F_\gamma(t) \le \max\{ (F_\sigma(t))^2, (F_\gamma(t))^2 \}\]
Hence
\[ F_{\max\{\sigma,\gamma\}}^{(-1)}(u) \le \max\{ F_{\max\{\sigma_1,\sigma_2\}}^{(-1)} (u), F_{\max\{\gamma_1,\gamma_2\}}^{(-1)} (u)  \} \le F_{\max\{\sigma_1,\sigma_2\}}^{(-1)} (u)+ F_{\max\{\gamma_1,\gamma_2\}}^{(-1)} (u)    \]
Since the rhs is integrable as in \eqref{eq::explosive-crit-2}, so is the left-hand side. Hence, the process $(h_X, F_{\max\{\sigma, \gamma\}}, 0, \infty)$ is also explosive. 
Clearly  $\min\{\sigma, \gamma\}\ {\buildrel d \over \le }\ \max\{\sigma, \gamma\}$ and hence by Theorem \ref{thm::comp1} $(h_X, F_{\min\{\sigma, \gamma\}}, 0, \infty)$ explodes as well.
For the sum, it holds that $\sigma+\gamma \le 2 \max\{ \sigma, \gamma\}$, and hence
\[ F_{\sigma+\gamma}^{(-1)}(u) \le F_{\max\{\sigma,\gamma\}}^{(-1)}(u/2).\]
Integrability of the rhs  as in \eqref{eq::explosive-crit-2} implies integrability of the lhs and so the process $(h_X, F_{\sigma+\gamma}, 0, \infty)$ is also explosive. 
\end{proof}
%
\section{A new proof of the `sufficient part' of Theorem \ref{thm::amidev2}}\label{s::newproof}
In this section, we give a new proof for the sufficient part of Theorem \ref{thm::amidev2} for plump power-law distributions. Note that the proof in \cite{AmiDev13} holds for all plump distributions not just plump power-laws, so the proof presented here is less general. However, since in nature it is quite different,  it reveals a different aspect of explosive branching processes.
It provides a thinning argument and thus it enables us to apply it later for processes with incubation periods as well.

The idea of the proof originates  from the work of \cite{Sev67} where he applied this thinning method on the BP, however, with different choices of retain probabilities, obtaining thus weaker results.
Thus, in this section we re-prove the `sufficient' direction of Theorem \ref{thm::amidev2}. This is the harder direction for plump power-laws.

\begin{theorem}[Weaker version of Theorem \ref{thm::amidev2}]\label{thm::amidev-re}
Given a plump offspring distribution $X$ satisfying \eqref{eq::plump}, and a weight distribution $F_\sigma$, the branching process $(h_X, F_\sigma, 0, \infty)$ is min-summable and hence also explosive if there exists an $\ve>0$ and $C>0$ such that
\be\label{eq::explosive-crit-4}\int_{1/\ve}^{\infty} F_\sigma^{(-1)} \left( \e^{-Cu} \right) \frac{1}{u}\mathrm du<\infty.\ee
This criterion is necessary and sufficient for the explosivity of the process when we require that $X$ is a plump \emph{power-law} distribution, as in \eqref{eq::plump2}.
\end{theorem}

Before the proof, we need some preparation.
Recall that the date of birth of an individual $x$ in the BP is denoted by $\tau_x$, while $\sigma_x=\tau_x-\tau_{x_{|-1}}$, the birth-time of $x$, denotes the time difference between the birth date of the individual and its mother.

\begin{definition}\label{def::thinning}
Consider an age-dependent BP with offspring distribution $X$ and birth time distribution $\sigma$. Fix a sequence $(t_n)_{n\in \N}$ with $\sum_{i=1}^\infty t_n:=T<\infty.$
Let us do a generation dependent thinning on the BP as follows: we keep all the offspring of the root and denote their number by $\eta_1$. We delete the whole subtree of an individual $x\in G_1$ if $\sigma_x>t_1$, and we denote the total number of individuals that are not deleted in $G_2$ by $\eta_2$. Similarly, we delete the whole subtree of an individual $x \in G_n$ if $\sigma_n>t_n$ and denote the total number of individuals in $G_{n+1}$ that have not been deleted in any of the previous steps by $\eta_{n+1}$. Let us denote the sub-branching process of non-thinned vertices by $BP_\eta$.
\end{definition}
The next claim is elementary.
\begin{claim}\label{cl::Wn}
Let $H_n(s):=\Ev[s^{\eta_n}]$ be the generating function of $\eta_n$, and let us write $p_n:=F_\sigma(t_n)$ for the retention probability of a subtree of a vertex in $G_n$. 
Then $W_n(s)=1-H_n(1-s)$ satisfies the recursion
\be\label{eq::wn-recursion} W_{n+1}(s)= W_n(p_n g(s))\ee
with $g(s):=1-h_X(1-s)$, $W_0(s):=s$ and $p_0:=1$.
For the power-law distribution $h_\al(s):=1-(1-s)^\al$,
\be\label{eq::wn-alpha} W^{(\al)}_{n}(s)= p_1^\al p_2^{\al^2} \dots p_{n-1}^{\al^{n-1}} s^{\al^n}=s^{\al^n} \prod_{i=1}^{n-1} p_{i}^{\al^i}.\ee
\end{claim}
\begin{proof}
If we keep a vertex in $G_n$, because $\sigma<t_n$, then all the children of that vertex belong to $\eta_{n+1}$. Hence
the generating function $H_{n+1}$ of $\eta_{n+1}$ satisfies
\[ H_{n+1}(s)=\Ev[ \left( p_n h_X(s) + 1-p_n \right)^{\eta_n}] =H_n(1- p_n(1-h_X(s))).\]
Then, with $g(s)=1-h_X(1-s)$ and $W_{n}(s)=1-H_n(1-s)$ it is elementary to derive \eqref{eq::wn-recursion}. The values of $W_0$ and $p_0$ are choices that correspond to the proper initialisation. The second statement, \eqref{eq::wn-alpha} is then elementary calculation, since $1-h_\al(1-s)=s^\al$.
\end{proof}

The next lemma is the core of the proof of Theorem \ref{thm::amidev-re}.
\begin{lemma}\label{lem::sufficient}
Suppose that $\sigma$ satisfies the integrability criterion \eqref{eq::explosive-crit-2} in Lemma \ref{lemma::integral} and $X$ is a plump distribution with
lower bound $1-\delta\in (0,1)$ on the power-law exponent, as in \eqref{eq::plump}. Let the
 retention probabilities be defined as $p_n=1/\exp\{-C/(1-\delta/4)^n\}$. 
  Then, for a sufficiently large $C>0$, the thinned process $BP_\eta$ of $(h_X, F_\sigma, 0, \infty)$ is supercritical in the sense that $\lim_{n\to \infty} \Pv(\eta_n= 0)=\Pv(BP_\eta \text{\ goes extinct\,}) <1$. Further, all the individuals in $BP_\eta$ can be reached in finite time from the root. \end{lemma}

\begin{proof}
Since $F_\sigma$ satisfies the integrability criterion \eqref{eq::explosive-crit-2}, following the proof of  Lemma \ref{lemma::integral}, (in particular the bounds in \eqref{eq::bound11}, \eqref{eq::bound12}), the series
\[ \sum_{i=1}^\infty F_\sigma^{(-1)} \left( \exp\{-C/ \beta^n\} \right):=T(\beta, C)<\infty\]
 for all $C>0$ and $\beta\in(0,1)$.

 Since $X$ is a plump distribution, the lower bound in \eqref{eq::plump} holds with some $\delta>0$. So, by Karamata's Tauberian theorem \cite[Theorem 1.7.1]{BiGoTe87}, 
$1-h_X(s)\ge (1-s)^{1-\delta}$ if $s$ is sufficiently close to $1$. This means that we can set $\alpha:=1-\delta/2$, and define $Y_\al$ with $h_Y(s)=1-(1-s)^\al$ such that $h_X(s)\le h_Y(s)$ holds whenever $s \in [s_0, 1]$, for some $s_0<1$. This further implies $g(s)=1-h_X(1-s)\ge s^\al$ in $[0, 1-s_0]$. Set $\beta:=1-\delta/4$, and
 $C>0$ so large that 
 \be\label{eq::p1crit} \exp\left\{- \frac{C}{\beta} \frac{1}{1-\al/\beta} \right\} <1-s_0\ee
 holds. Further, we choose the thinning tresholds to be $t_n:=F_\sigma^{(-1)} \left( \exp\{-C/ \beta^n\} \right)$ yielding that the retention probability of a subtree of a vertex in $G_n$ is $p_n=F_\sigma(t_n)=\exp\{-C/\beta^n\}$. We apply the thinning to the BP $(h_X, F_\sigma, 0, \infty)$ as described in Definition \ref{def::thinning}.
 Note that in this case the function $W_n$ defined in Claim \ref{cl::Wn} can be bounded from below by 
 \[ W_{n+1}(s)=W_n(p_n g(s))\ge W_n(p_n s^\al)\]
 as long as $s\in [0, 1-s_0]$.
 Continuing the recursion for $k$ steps, we obtain
 \[W_{n+1}(s)\ge W_{n-k}(p_{n-k} p_{n-k+1}^\al \dots p_n^{\al^k} s^{\al^{k+1}}) \]
 and we can continue to lower bound $W_{n-k}$ by using $W_{n-k-1}$ as long as
 \be\label{eq::cond-s0} s^{\al^{k+1}}\prod_{j=1}^k p_{n-k+j}^{\al^j} \in [0, 1-s_0] \ee
 holds,
 since in this case the argument of $g(\cdot)$ in the next step will be again in $[0, 1-s_0]$ and so the lower bound on $g$ is still valid.
  We calculate using that $p_n=\exp\{-C/\beta^n\}$, and that $\al<\beta$,
 \[ \prod_{j=1}^k p_{n-k+j}^{\al^j}= \exp\Big\{ -\frac{C}{\beta^{n-k}} \sum_{j=1}^{k} (\al/\beta)^j\Big\} \le \exp\left\{- \frac{C}{\beta^{n-k}} \frac{1}{1-\al/\beta}\right\}.\]
 As a result, assuming that $C$ is large enough for  \eqref{eq::p1crit} to hold, \eqref{eq::cond-s0} holds for all $k$ and $s \in [0, 1]$. Thus, we can apply the recursive step $n$ times to obtain that
  \[ W_{n+1}(s) \ge s^{\al^{n+1}}\prod_{j=1}^n p_{j}^{\al^j}=W_{n+1}^{(\al)}(s) \]
  holds for all $s \in [0, 1]$ (see \eqref{eq::wn-alpha}).
 
Now we show that the thinned process $BP_\eta$ is supercritical with these thinning probabilities. 
Note that $W_n(1)=\Pv(\eta_n>0)$, is a decreasing sequence of $n$ and as a result $\lim_{n\to \infty}W_n(1):=W(1)=\Pv(\eta_\infty>0)$ exists and equals the survival probability of the sub-branching process $BP_\eta$. Thus, if we show that $W(1)=\lim_{n\to \infty} \Pv(\eta_n \neq 0)>0$ then the thinned process $BP_\eta$ is \emph{supercritical} in the sense that it survives forever with positive probability. 

Finally, all non-deleted individuals in $G_n$ are accessible from the root by a path of length at most $\sum_{i=1}^{n-1} t_i<T(\beta, C)<\infty$. 
Note that in this case 
\[ F(T,s):=\Ev[s^{\sum_{x\in BP}\ind_{\tau_x\le T}}]\]
satisfies that $F(T, 1)<1$, that is, the process is explosive. Indeed, with probability $W(1)$, there is at least one infinite ray $x\in G_\infty \cup \CD(T)$ and as a result $D(T)=\infty$ with probability $W(1)$.

It is left to show that $\lim_{n\to \infty} W_n(1)>0$. 
For this we calculate the lower bound
\[ W_n(1) \ge W_n^{(\al)}(1)= \prod_{j=1}^{n-1} p_{j}^{\al^j} = \exp\Big\{- \frac{C}{\beta} \sum_{j=0}^{n-1} (\al/\beta)^j\Big\}\to  \exp\Big\{ \frac{C}{\beta} \frac{1}{1-\al/\beta} \Big\}>0,\]
since $\al/\beta<1$. 
\end{proof}
\begin{remark}\normalfont
The core of this thinning argument is that the integral criterion in Lemma \ref{lemma::integral} does not depend on the precise value of $C$ or on the powers $\beta$ (or $\alpha$). This is reflected in the fact that the explosion is robust in the sense that if a lifetime distribution $\sigma$ explodes for one value of power-law $\al$ then it explodes for all values, and hence a thinner-tailed BP still explodes. This is essentially the argument also used in the proof of Theorem \ref{thm::amidev2} in the paper \cite{AmiDev13}. There, they develop an algorithm that thins the BP in a degree-dependent way. The thinning leaves a sub-BP in which the degrees grow according to a thinner tail power-law, and that BP is still explosive. 
\end{remark}
\begin{proof}[Proof of Theorem \ref{thm::amidev-re}]
Suppose $X$ is a plump distribution and $\sigma$ satisfies the integrability criterion in Lemma \ref{lemma::integral}. Lemma \ref{lem::sufficient} then shows that the integral criterion is sufficient for the explosion of the process. To finish the proof, we need to show that once $X$ follows a plump \emph{power-law}, 
it is necessary as well. For this we use the upper bound on the distribution function of $X$, as in \eqref{eq::plump2}. 
This bound ensures that one can couple $X$ to  a random variable $Y$ that has tail probabilities $\Pv(Y\ge x)=C/x^\delta$ such that $X$ is stochastically dominated by $Y$. And/or, one can use a Tauberian theorem to say that $h_X(s)\ge 1-(1-s)^{\delta}$ in a sufficiently small neighborhood of $1$. Either way, by comparison Theorem \ref{thm::comp-tail}, the explosivity of the BP $(h_Y, F_\sigma, 0, \infty)$  is necessary for the explosivity of $(h_X, F_\sigma, 0, \infty)$.

We argue that the latter is conservative when  the integral criterion \eqref{eq::explosive-crit-3} is not met. Here the easiest argument is the same as the one in the proof of Theorem \ref{thm::amidev2} in \cite{AmiDev13}: the size $Z_k$ of generation $k$ of the Galton-Watson BP  with offspring distribution $Y$ grows double-exponentially. More precisely, by \cite{Dav78}, there exists a random variable $V$ such that  
\[ \delta^k \log (\min\{1, Z_k\}) \toas V, \]
and $V>0$ on survival.
As a result, for some $K>0$, for all $k\ge K$, $Z_k\le \exp\{ 2 V \delta^{-k}\}$ holds for some $V>0$. 
Now we add the i.i.d.\ edge-weights to the edges of the GW tree, and set $a_k:= \exp\{ 2 V \delta^{-k}.\}$
Clearly, the process is conservative if the sum of the minimum edge-weight in each generation diverges. 
This minimum is then at least 
\be\label{eq::min-sum-2} \sum_{k=K}^{\infty} \min\{\sigma_{k,1}, \sigma_{k,2}, \dots,  \sigma_{k, a_k}\},\ee
since in each generation there are at most $a_k$ individuals. Thus, divergence of the sum in \eqref{eq::min-sum-2} ensures that the BP $(h_X, F_\sigma, 0, \infty)$ is conservative. This sum is divergent precisely when the sum $\sum_{k=K}^\infty F_\sigma^{(-1)}(1/a_k)$ diverges: for this see the proof of \cite[Corollary 4.3]{AmiDev13}. Finally, the equi-convergence of this sum and the integral in Lemma \ref{lemma::integral} is precisely the content of Lemma \ref{lemma::integral}.
\end{proof}

\section{Processes with incubation period}\label{s::incu}
We have seen in Theorem \ref{thm::incu-matter} that the explosivity of both $(h_X, F_\sigma, 0, \infty)$ and $(h_X, F_I, 0, \infty)$ is necessary for the explosivity of the incubation model $(h_X, F_\sigma, F_I, \infty)$. In this section we aim to show the reverse direction, i.e., that it is also sufficient.
The crucial problem to overcome is the following: by superimposing an incubation time on the branching process, exactly the \emph{short} edges are killed.
In fact, the shorter the edge the more likely that it is deleted when incubations come into the picture. This means that the thinning of the BP is not an independent binomial thinning, even in the case of the (simpler) backward process, where the birth-times of children are independent. 
 
 The outline of the proof is the following: first we prove that the backward process explodes, given the explosivity of the two BP-s  $(h_X, F_\sigma, 0, \infty)$ and $(h_X, F_I, 0, \infty)$. The proof is based on showing that for the new birth-time distribution, the summability criterion in \eqref{eq::explosive-crit-1} (integral criterion) stays valid. 
 
 For the forward process, we use a different argument: we develop a generation based thinning of the BP, similar to the one in the proof of Theorem \ref{thm::amidev-re}.
 
Our first goal is thus to prove the following:
\begin{theorem}\label{thm::incu-main}
Consider an epidemic model with offspring distribution $X$ that follows a plump distribution, birth-time distribution $\sigma$ and incubation time distribution $I$. Then, the \emph{backward} process $(h_X, F_\sigma, F_I, \infty)^b$ explodes \emph{if and only if} both the processes $(h_X, F_\sigma, 0, \infty)$ and $(h_X, F_I, 0, \infty)$ explode. 
\end{theorem}
Then, we will show the somewhat weaker 
\begin{theorem}\label{thm::incu-main-forw}
Consider an epidemic model with offspring distribution $X$ that follows a plump power-law distribution, birth time distribution $\sigma$ and incubation time distribution $I$. Then, the forward process $(h_X, F_\sigma, F_I, \infty)$ explodes \emph{if and only if} both the processes $(h_X, F_\sigma, 0, \infty)$ and $(h_X, F_I, 0, \infty)$ both explode, that is, if and only if they satisfy the integrability criterion in \eqref{eq::explosive-crit-1}. This integrability condition is sufficient for explosion of $(h_X, F_\sigma, F_I, \infty)$ when the offspring distribution $X$ is plump.
\end{theorem}

\begin{remark}\normalfont
This theorem is somewhat weaker since we lost the `only if' direction when the offspring distribution is plump, i.e., it has strictly heavier tails than any power-law. In fact, the proof of Theorem \ref{thm::incu-main-forw} below could also work to prove Theorem \ref{thm::incu-main}, but again, then we would lose the necessity for plump distributions that are not plump power-laws.
\end{remark}

In the master thesis of L. Gulikers \cite{Gul14}, Gulikers and the author provided a proof of Theorem \ref{thm::incu-main} and Theorem \ref{thm::incu-main-forw} under different conditions. There, we proved that the explosion of $(h_X, F_\sigma, 0, \infty)$ and $(h_X, F_I, 0, \infty)$ is sufficient for the \emph{backward} epidemic model $(h_X, F_\sigma, F_I, \infty)^b$ to explode under the more restrictive assumptions that $X$ has a power-law distribution with parameter $\alpha\in (0,1)$ and 
either

1) $F_I(t)>F_\sigma(t)$ in some open interval around $0$, or

2) $F_I$ and $F_\sigma$ are absolutely continuous and the densities $f_I, f_\sigma $ satisfy $f_I(t) \le f_\sigma(t)$ in some open interval around $0$.

Further, we showed that the explosion of the backward epidemic model $(h_X, F_\sigma, F_I, \infty)^b$ implies the explosion of the \emph{forward} epidemic model $(h_X, F_\sigma, F_I, \infty)$ when $F_\sigma$  is an \emph{ageing distribution} in some neighborhood of the origin, that is, for some $t_0>0$, for all $t\in [0, t_0]$ it holds that

\[ (\sigma -t | \sigma \ge t)\  {\buildrel d,0 \over \le }\  \sigma.\]
In particular, this condition is satisfied if $F_\sigma$ has a density $f_\sigma$ with $f_\sigma(0)=0$ and $f_\sigma$ non-decreasing in some neighborhood of the origin. Here we give a counterexample below in Example \ref{ex::counter}: an absolutely continuous distribution, with full support, that is locally non-monotonous at the origin yet it produces an explosive BP.

Before we start, we need some preparation and recall some basic theory about regularly varying functions.
\begin{definition}\label{def::slowvar}
We say that a function $\ell(x)$ is \emph{slowly varying at $0$} if $k(x):=\ell(1/x)$ is slowly varying at infinity, that is, if for all $\lambda>0$,
\[ \lim_{x\to 0} \frac{\ell(\la x)}{\ell(x)}=\lim_{x \to \infty} \frac{k(\la x)}{k(x)} =1.\]
\end{definition}
We say that two functions $f,g$ are asymptotically equivalent if $\lim_{x\to \infty} f(x)/g(x)=1$, and they are  asymptotically equivalent at $0$ if 
$\lim_{x\to 0} f(x)/g(x)=1$. 

The following theorem gives the characterisation of slowly varying functions, by  Karamata \cite{Kar30} that can be found as \cite[Theorem 1.3.1]{BiGoTe87}. 

\begin{theorem}[Karamata]\label{thm::smooth}
Let $k(x)$ be a slowly varying function at infinity. Then $k(x)$ has the representation
\be\label{eq::slow-rep} k(x)=\exp\left\{ c(x) + \int_a^x \frac{\varepsilon(t)}{t} \mathrm dt \right\}\ee
with $\lim_{x\to \infty}c(x)=c$, and $\lim_{x\to \infty}\ve(x)\to 0$. Further, $\ve(x)$ can be chosen to be arbitrarily smooth, and $\ve(x)$ is eventually negative if $\lim_{x\to \infty} k(x) = 0$.
\end{theorem}
This theorem yields the following corollary:
\begin{corollary}\label{cor::slowly-0-prop}
Let $\ell(x)$ be a slowly varying function at $0$ with $\ell(x)=0$. Then $\ell(x)$ has the representation
\be\label{eq::slow-rep2} \ell(x)=\exp\left\{ c(x) + \int_a^{1/x} \frac{\varepsilon(t)}{t} \mathrm dt \right\}\ee
with $\lim_{x\to 0}c(x)=c$, and $\lim_{x\to \infty}\ve(x)\to 0$. 
Further, $\ve(x)$ is eventually negative with $\ve(x)\to 0$ but $\int_0^\infty \ve(t)/t \mathrm dt =-\infty$, hence for all $\beta>0$
\be\label{eq::epsilon-dom} \lim_{t\to \infty} |\ve(t)| t^{\beta} =\infty. \ee
\end{corollary}

\begin{proof}
The representation of $\ell(x)$ is obvious from the representation of $k(x)$ in \eqref{eq::slow-rep} by noting that since $c(x)\to c$ in the representation of $k(x)$, we can re-index this function in the representation of $\ell$. Further, when $\ell(0)=0$ then $\lim_{x\to \infty} k(x)=0$ is necessary so $\ve(t)$ is eventually negative and $\int_0^\infty \ve(t)/t \mathrm dt =-\infty$. Formula \eqref{eq::epsilon-dom} is obvious from the divergence of the integral.
\end{proof}

An elementary property of slowly varying functions is that they are sub-polynomial, i.e., for all $\al>0$, $\lim_{x\to \infty} k(x)x^\al=\infty$. They also decrease rather slowly: for all $\al>0,\, \lim_{x\to \infty} k(x)x^{-\al}=0$ . These translates to
the fact that all slowly varying functions at $0$ with $\ell(0)=0$ are steeper than any polynomial, but tend to $0$ when multiplied by one, i.e., for all $\al>0$,
\be\label{eq::steep} \lim_{x\to 0} \frac{\ell(x)}{x^{\al}}=\infty, \quad \lim_{x\to 0} \ell(x)x^{\al}=0. \ee

\begin{proof}[Proof of Theorem \ref{thm::incu-main}]
The explosion of both $(h_X, F_\sigma, 0, \infty)$ and $(h_X, F_I, 0, \infty)$ is necessary by Theorem \ref{thm::incu-matter}. We yet have to prove that it is sufficient.
First note that in the backward process, the distribution of the birth-times across the offspring is i.i.d. from the non-regular distribution function
\[ F_{I<\sigma}(t):=\int_0^t F_I(x) F_\sigma(\mathrm dx) \]
As a result, we can consider the backward process as an age-dependent BP with this new birth-time distribution (that has a positive probability to be infinite).
With the usual notation, we need to show that $(h_X, F_{I<\sigma}, 0, \infty)$ explodes as well.
For this, we shall use the fact that multiplying by a constant, or taking the power of a distribution function does not change the summability of \eqref{eq::explosive-crit-1}. We have shown this only for plump power-laws using the integrability in Lemma \ref{lemma::integral}, in the proof of  Theorem\ref{thm::max-birthtime}. However, it is not hard to see that these properties also remain valid when $X$ has heavier tails then power-laws, e.g. when $\Pv(X>x)= C/\lfloor \log x\rfloor$. For simplicity we use the integral characterisation valid for plump power-laws as in Lemma \ref{lemma::integral}.

We start by partial integration: 
\[ F_{I<\sigma}(t)= \int_0^t (F_\sigma(t)-F_\sigma(y)) F_I(\mathrm dy) \ge (F_\sigma(t)-F_\sigma(a t)) F_I(a t) \]
for any $a \in (0,1)$.
Suppose now that 
\be\label{eq::assump}\exists a \in (0,1), q \in [0, 1): \quad \limsup_{t\to 0} \frac{F_\sigma(a t)}{F_\sigma(t)} \le q <1. \ee
Then for a small enough $t_0>0$, for all $t \in (0, t_0]$  
\be\label{eq::Fisigma-lower}F_{I<\sigma}(t) \ge (1-q) F_\sigma(t) F_I(at). \ee
Using the same argument as in the proof of Theorems \ref{thm::max-birthtime} we obtain that the inverse function of $F^{(-1)}_{I<\sigma}(z)$ is also integrable in the sense of Lemma \ref{lemma::integral} when $F_\sigma^{(-1)}$ and $F_I^{(-1)}$ are both integrable.

1) When \eqref{eq::assump} holds then the previous argument proves the statement of the theorem.
Note that this covers the most important `borderline' cases such as $F_\sigma(x)=\exp\{- e^{C/x^\gamma}\}$ or $F_\sigma(x)=\exp\{- C/x^\gamma\}$ but also all polynomials, that is, when $F_\sigma$ is regularly varying at the origin:
\be\label{eq::poli} F_\sigma(t) =t^\beta \ell(t)\ee
holds for some arbitrary $\beta>0$  and slowly varying function $\ell$ in some neighborhood of the origin.

2) When \eqref{eq::assump} does not hold then
\be \label{eq::no-assump}\forall a \in (0,1):\quad \lim_{t\to 0} \frac{F_\sigma(a t)}{F_\sigma(t)}=1,\ee
where we replaced the $\limsup$ with a $\lim$ since $F_\sigma$ is monotonously increasing.
Further, note that \eqref{eq::no-assump} means precisely that $F_\sigma$ is \emph{slowly varying at $0$}, in other words, $G(x):=F_\sigma(1/x)$ is slowly varying at infinity. Plus, since $F_\sigma(0)=0$, $\lim_{x\to \infty}G(x)=0$ must hold as well. 
As a result, by \eqref{eq::steep}, we obtain that in this case $F_\sigma$ is \emph{steep} at the origin. Examples include $F_\sigma(x)=1/\log^\al (1/x)$ for some $\al>0$ or $\exp\{-\log^\al (1/x)\}$ for $0<\al<1$.

It is thus natural to prove the explosivity using Theorem \ref{thm::comp4}: If we can find an increasing function $H$ in an interval $[0, t_0]$ that satisfies 
$F_\sigma(\mathrm dt) \ge H(\mathrm dt)$ and $(h_X, H, F_I, \infty)^b$ is explosive, then so is $(h_X, F_\sigma, F_I, \infty)^b$.

So, let us consider the function
\be\label{eq::Hgamma} H_\gamma(x):=x^\gamma F_\sigma(x)\ee
in some small neighborhood of the origin for some $\gamma>0$. 
Then, using the representation in \eqref{eq::slow-rep2} with smooth $\ve(\cdot)$, 
\[  \ba H_\gamma(\mathrm d x)&= F_\sigma(x) \left( x^{\gamma-1} (\gamma- \ve(1/x)) \mathrm dx  +x^\gamma c(\mathrm dx) \right),\\
F_\sigma(\mathrm d x)&= F_\sigma(x) \left( - \ve(1/x)/x \,\mathrm dx  +c(\mathrm dx) \right)
.\ea \]
Note that for all $x\in(0,1)$, $x^\gamma c(\mathrm dx)< c(\mathrm dx)$ so the possible non-smooth part in $F_\sigma$ is decreased.
For the other term, we need $x^{\gamma} (\gamma+|\ve(1/x)|)< | \ve(1/x)|$ where recall from Corollary \ref{cor::slowly-0-prop} that $\ve$ is eventually nonpositive. Clearly  $x^{\gamma} \gamma< | \ve(1/x)|(1-x^\gamma)$ holds since $1-x^\gamma>1/2$ for $x$ small enough, and then  $\ve(1/x)/x^\gamma = \ve(z) z^{\gamma}$ tends to infinity with $z=1/x \to \infty$ by formula \eqref{eq::epsilon-dom}.   

It remains to show that $(h_X, H_\gamma, F_I, \infty)$ is still explosive. For this, first note that $H_\gamma$ satisfies the integrability criterion in Lemma \ref{lemma::integral}, since by \eqref{eq::steep}, for all small enough $x$ and any $\beta>0$, $H_\gamma(x)\ge x^{\gamma+\beta}$ and hence $H_\gamma^{(-1)}(z)\le z^{1/(\gamma+\beta)}$. This function is clearly integrable as in \eqref{eq::explosive-crit-3} so $(h_X, H_\gamma, 0, \infty)$ is explosive. If $X$ is a plump distribution but not a plump power-law then $X$ has heavier tails then a plump power-law and hence $H_\gamma$ certainly satisfies the summability criterion \eqref{eq::explosive-crit-1}.
To finish, we  return to case 1): clearly $H_\gamma$ satisfies \eqref{eq::assump} since $H_\gamma$ is regularly varying with index $\gamma$ at $0$, hence
\[ \forall a \in (0, 1):\quad \lim_{x\to 0}  \frac{H_\gamma(a x)}{H_\gamma(x)}=a^\gamma<1.\]
As a result, $(h_X, H_\gamma, F_I, 0)^b$ is explosive. This implies that $(h_X, F_\sigma, F_I, 0)^b$ is explosive too.
 \end{proof}
\begin{proof}[Proof of Theorem \ref{thm::incu-main-forw}]
Necessity follows from Theorem \ref{thm::incu-matter}. It is enough to show that the explosivity of both processes is sufficient. 
By the same argument as in the proof of Theorem \ref{thm::amidev-re}, $h_X(s)\le 1-(1-s)^{\delta}$ in some small neighborhood of $1$. Thus, let us define $Y$ as a random variable with generating function $h_Y(s)=1-(1-s)^\al$ with $\al=\delta$. By applying comparison Theorem \ref{thm::comp-tail} for the BP-s with incubation times and Corollary \ref{cor::power-law-change} for the two age-dependent BPs, it is enough to show that the forward process $(h_X, F_\sigma, F_I, \infty)$ explodes whenever both $(h_Y, F_\sigma, 0, \infty)$ and $(h_Y, F_I, 0, \infty)$ explode.

Next, we modify $F_\sigma$ to $H_\gamma$ from \eqref{eq::Hgamma} when necessary, so that $F_\sigma$ satisfies \eqref{eq::assump}. That is, when $F_\sigma$ varies slowly at $0$ then we rather consider $H_\gamma(t)=F_\sigma(t) t^\gamma$. In this case, since $F_\sigma(\mathrm dt) \ge H_\gamma (\mathrm dt)$ for all $t\in [0,t_0]$, by Theorem \ref{thm::comp4}, the explosivity of $(h_Y, H_\gamma, F_I, \infty)$ implies the explosivity of $(h_Y, F_\sigma, F_I, \infty)$. Clearly the process $(h_Y, H_\gamma, 0, \infty)$ is also explosive, since it satisfies the criterion \eqref{eq::explosive-crit-2}.

Thus, we reduced the problem for showing that $(h_Y, F_\sigma, F_I, \infty)$ is explosive whenever $(h_Y, F_\sigma, 0, \infty)$ and $(h_Y, F_\sigma, 0, \infty)$ explode \emph{and} the condition in \eqref{eq::assump} is met.

First, we modify the proof of Theorem \ref{thm::amidev-re} as follows: we develop a similar thinning as in Definition \ref{def::thinning}.  Indeed, for the forward process, our new thinning will look as follows: 

We fix a sequence $\wit t_n$ with $\sum_{n=1}^\infty \wit t_n=\wit T<\infty$ and we fix an $a\in (0,1)$ from \eqref{eq::assump}. We keep the (sub)tree of the root only if its incubation time $I_0< a \wit t_1$. We denote the kept vertices in generation $1$ by $\wit\eta_1$. Next, we keep the whole subtree of a vertex $x$ in $G_1$ if and only if both $I_x<a \wit t_2$ and $\sigma_x\in [a \wit t_1, \wit t_1]$. We denote the number of generation-$2$ individuals that are kept by $\wit\eta_2$, and so on. That is, we keep the whole subtree of a vertex $x$ in generation $n$ if and only if $I_x\le a\wit t_{n+1}$ and $\sigma_x \in [a \wit t_n, \wit t_n]$. We denote the vertices kept in generation $n$ by $\wit\eta_n$. Finally, we denote the subtree of kept individuals by $BP_{\wit\eta}$. 

The crutial idea of this thinning is that for all the vertices in $BP_{\wit\eta}$, the incubation time of the parent vertex is \emph{shorter} than the birth-time of the individual. Indeed, for a vertex $x$ in $\eta_n$, the parent $x_{|-1}$ of this vertex must be in $\eta_{n-1}$ and thus it must have $I_{x_{|-1}}<a \wit t_{n}$, while, since the vertex $x$ is in $\eta_n$, it must have $\sigma_x \in [a \wit t_n, \wit t_n]$. As a result, in the forward process of the epidemic with incubation times, all the vertices in $BP_{\wit\eta}$  will be infected.

The advantage of this thinning is that it leaves an i.i.d.\ thinning on the subtrees of vertices, just as before in the proof of Theorem \ref{thm::amidev-re}:
In this case, the retention probability of a subtree of a vertex in generation $n$ is, by the condition in \eqref{eq::assump}
\[ F_I(a \wit t_{n+1}) (F_\sigma(\wit t_n)- F_\sigma(a \wit t_n))\ge (1-q)F_I(a \wit t_{n+1}) F_\sigma(\wit t_n)  \ge (1-q)F_I(a \wit t_{n+1}) F_\sigma(\wit t_{n+1}):=\wit p_n.  \]

It is left to show that, for a proper choice of $\wit t_n$, with $\sum_{n=1}^\infty \wit t_n <\infty$,
 the thinned process $BP_{\wit\eta}$ is supercritical again in the sense that $\lim_{n\to \infty} \Pv(\wit\eta_n =0)<1$. This follows from the following two facts: 
 
1) Define the distribution function $\wit F(t):=(1-q) F_I(\delta t) F_\sigma(t)$. Then $\wit F(\wit t_{n+1})=\wit p_n$. This distribution function satisfies the integrability criterion of Lemma \ref{lemma::integral}, provided both $F_\sigma$ and $F_I$ do so as well. The proof of this claim is immediate from the proof of Theorem \ref{thm::max-birthtime}, i.e., the proof that the maxima of two distributions also explode.

2) Theorem \ref{thm::amidev-re} ensures that setting $t_{n}=\wit F^{(-1)}(\exp\{C/\beta^{n}\})$ for a sufficiently large $C>0$ and $\beta>\al$ yields a thinning of $(h_Y, \wit F, 0, \infty)$ with retention probability $\wit p_n= \exp\{C/\beta^{n}\}$ that produces a supercritical thinned BP. Finally, by possibly modifying $C$ when necessary, is not hard to see that thinning the BP $(h_Y, \wit F, 0, \infty)$ by shifted indices $\wit t_{n}:=t_{n+1}$ and $\wit p_n:=p_{n+1}$ also produces a supercritical thinned BP.  This finishes the proof. 
\end{proof}

\section{Some educational examples}\label{s::edu}
In this section we investigate some interesting birth-time distributions. We construct a class of birth-time distributions singular to the Lebesque measure that form an explosive BP with plump power-law offspring distributions. We also investigate the `borderline' birth-time distribution for explosivity. 

We start with a motivating example from singular distributions. The following construction is known from fractal theory. 
\begin{example}[Natural measure on the Cantor-set]\normalfont
Consider the distribution function of the natural probability measure on the Cantor set, that is obtained as follows: the $n$th approximation of the Cantor set consist of the union of $2^n$ many intervals, each of length $3^{-n}$:
\be\label{eq::cantor-n} \bigcup_{x_1,\dots, x_n: x_i\in \{0,2\}} \left[ \sum_{i=1}^n \frac{x_i}{3^i}, \frac{1}{3^n}+\sum_{i=1}^n \frac{x_i}{3^i}\right] \ee
Consider $F_n(x)$ as the distribution function of the uniform measure on this $n$th approximation i.e., where each interval has measure $2^{-n}$. It is not hard to show that the distribution function $\lim_{n\to \infty}F_n(x)= F_{\text{Cantor}}(x)$ converges pointwise and the limit function $F_{\text{Cantor}}(x)$ is a continuous, monoton function that only increases on the Cantor set, and is constant otherwise. Thus, $F_{\text{Cantor}}(x)$ is the distribution function singular to the Lebesque measure.
\end{example}
The next example is another singular distribution function, that is precisely the inverse of $F_{\text{Cantor}}(x)$: it is a discreet measure of the length of the complement of the Cantor set:
\begin{example}\normalfont Consider the dyadic expansion of $x\in(0,1)$, that is, $x=\sum_{i=1}^\infty x_i/2^i$, where $x_i\in\{0,1\}$. Let us define the following measure: $\mu_C(x):=1/3^n$ if and only if the last non-zero digit in the dyadic expansion of $x$ is at location $n$. Then, $\mu_C$ is a probability distribution singular to the Lebesque measure.
\end{example}
Indeed, there are $2^{n-1}$ many real numbers in $(0,1)$ that have measure $1/3^n$, and hence $\mu((0,1))=\sum_{n=1}^\infty 2^{n-1}/3^n=1$.
Here we show that both examples used as birth-time distributions yield explosive branching processes with any plump power-law offspring distribution.
\begin{claim} Let $X$ be a plump power-law distribution. Then $(h_X, F_{\text{Cantor}}, 0, \infty)$ and $(h_X, \mu_C, 0, \infty)$ are both explosive.
\end{claim}
\begin{proof}
We start with $(h_X, F_{\text{Cantor}}, 0, \infty)$. It is easy to see that for any $n\ge 1$, $F_{n+k}(1/3^n)=1/2^n$ for any $k\ge 0$. Hence, $F_{\text{Cantor}}(1/3^n)=1/2^n$ holds as well. As a result, for $u<1/2$,
\[ F^{(-1)}_{\text{Cantor}}(u) \le 3 u^{\log 3 /\log 2}.\]
Then
\be\label{eq::cantor-int} \int_0^{1/2} F^{(-1)}_{\text{Cantor}}(\e^{-1/y}) \frac1y \mathrm dy \le \int_0^{1/2} 3 \e^{-\frac{\log 3}{\log 2  y}} \frac1y \mathrm dy=
3\int_2^\infty \e^{-\frac{\log 3}{\log 2} z} \frac1z \mathrm dz <\infty.\ee
Lemma \ref{lemma::integral} finishes the proof. For the $\mu_C$, by e.g.\ using that it is the inverse of $F_{\text{Cantor}}(x)$, it is not hard to see that $\mu_C([0,1/2^n])=1/3^n$. The same calculation as in \eqref{eq::cantor-int} with $\log 2/\log 3$  in the exponent yields that this process is also explosive. \end{proof}
These examples motivated the following construction, that is an `almost' discreet distribution in the sense that its single accumulation point is $0$.
\begin{example}\label{ex::1}\normalfont Let the measure $\nu_\beta$, $\beta>1$ assign mass to the non-positive powers of $\e$: for $n\ge 1$, let $\nu_\beta(\e^{-n}):=\exp\{-\exp\{ \beta^n\}\},$ and let $\nu_\beta(1):=1-\sum_{n=1}^\infty \nu_\beta(\e^{-n}).$ Let $F_\beta$ denote the distribution function of the measure obtained. 
\end{example}
\begin{claim}
Let $X$ be a plump power-law distribution. Then $(h_X, F_\beta, 0, \infty)$ is explosive for $\beta<\e$ and conservative for $\beta\ge \e$.
\end{claim}
\begin{proof}
Note that for some constant $C>0$, for any $k\ge 1$, \[ \exp\{-\exp\{\beta^k\}\}<\sum_{n=k}^\infty \exp\{-\exp\{\beta^n\}\}\le C\exp\{-\exp\{\beta^k\}\}.\] Hence, $F_{\beta}(1/\e^n)\ge \exp\{ -\exp\{\beta^n\}\}$ and 
\[ F^{(-1)}_\beta(u) \le C (\log \log (1/u))^{-1/\log \beta} \]
As a result, 
\[ \int_{\e}^{\infty} F^{(-1)}_\beta(\e^{-u}) \frac1u \mathrm du=\int_{\e}^\infty (\log u)^{-1/\log \beta} \frac1u \mathrm du.\]
The latter integral converges if $1/\log \beta>1$ (that is, $\beta<\e$) and diverges if $1/\log \beta\le 1$, that is, if $\beta\ge \e$. Lemma \ref{lemma::integral} finishes the proof.\end{proof}
The next example is the continuous version of the previous example, with $\gamma=\log \beta.$ The proof is analogous and left to the reader. This example is important since for plump power-law distributed offsprings it is the \emph{boundary case} between explosive and conservative BPs. (Of course logarithmic corrections could be added.) 
\begin{example}[Continuous version of Example \ref{ex::1}]\label{ex::boundary}
Let $F_\gamma(y):=\exp\{ - \exp\{ 1/y^\gamma\}\}$. Then, for a plump power-law offspring distribution $X$, $(h_X, F_\gamma, 0, \infty)$ is explosive for $\gamma<1$ and conservative for $\gamma\ge 1$. 
\end{example}

We finish the paper by giving an example of a distribution that is absolutely continuous, but its density function is non-monotonous in any small neighborhood of the origin. This example shows an example for a birth-time distribution that does not satisfy the conditions in an older version of Theorem \ref{thm::incu-main-forw} in the Master thesis of Gulikers \cite{Gul14}. For a discussion about these conditions, see Section \ref{s::incu}.
\begin{example}[A counterexample]\label{ex::counter} \normalfont
Consider the following absolutely continuous measure: modify the singular distribution described in Example \ref{ex::1} for some $\beta<\e$, so that the mass $\exp\{-\exp\{ \beta^n\}\}$ is distributed uniformly over the interval $(\e^{-n}/2, 3\e^{-n}/2)$. Call the density of this part $f_\beta(t)$. Further, add  another  absolutely continuous part 
as in Example \ref{ex::boundary}, now for $\gamma>1$, with support on $[0,1]$, let us call the density of this part $f_\gamma(t)$. Then $f_\beta(t)+f_\gamma(t)$ describes the density function for all $t<1$. Add the remaining mass arbitrarily in an absolutely continuous way over some interval $(1, b]$ for some appropriate $b>1$. Call the resulting distribution function $F_\omega$.
\end{example}
This distribution function $F_\omega$ is absolutely continuous and has full support on $[0, b]$, plus, it forms an explosive BP with any plump-power law distribution, since already the first part (the modified version of Example \ref{ex::1}) is explosive. However, $F_\omega$ is \emph{not} monotone in any small neighborhood of the origin: for all small enough $t$, $f_\gamma(t)\ll f_\beta(t)$ and hence, for arbitrarily small $t_0$, the density $f_\gamma(t)+f_\beta(t)$ is non-monotonous on $[0, t_0]$.

\section*{Acknowledgements}
The work of JK was partly supported by the research programme Veni (project number 639.031.447), which is (partly) financed by the Netherlands Organisation for Scientific Research (NWO). Some results in this paper are generalisations of the results in the Master thesis of Lennart Gulikers, \cite{Gul14}, supervised by Remco van der Hofstad and the author. We refer the interested reader there for alternative proofs.  

\bibliographystyle{abbrv}
\bibliography{refsexplosion}

\begin{thebibliography}{10}

\bibitem{AmiDev13}
O.~Amini, L.~Devroye, S.~Griffiths, and N.~Olver.
\newblock On explosions in heavy-tailed branching random walks.
\newblock {\em The Annals of Probability}, 41(3B):1864--1899, 05 2013.

\bibitem{AN72}
K.~Athreya and P.~Ney.
\newblock {\em Branching Processes}.
\newblock Die Grundlehren der mathematischen Wissenschaften in
  Einzeldarstellungen. Springer Berlin Heidelberg, 1972.

\bibitem{Bara99}
A.-L. Barab{\'a}si and R.~Albert.
\newblock Emergence of scaling in random networks.
\newblock {\em science}, 286(5439):509--512, 1999.

\bibitem{bara00}
A.-L. Barab{\'a}si, R.~Albert, and H.~Jeong.
\newblock Scale-free characteristics of random networks: the topology of the
  world-wide web.
\newblock {\em Physica A: Statistical Mechanics and its Applications},
  281(1):69--77, 2000.

\bibitem{BaRe13}
A.~Barbour and G.~Reinert.
\newblock Approximating the epidemic curve.
\newblock {\em Electron. J. Probab.}, 18:54, 1--30, 2013.

\bibitem{BarHofKom15}
E.~Baroni, R.~v.~d. Hofstad, and J.~Komj{\'a}thy.
\newblock Fixed speed competition on the configuration model with infinite
  variance degrees: unequal speeds.
\newblock {\em Electron. J. Probab.}, 20:no. 116, 1--48, 2015.

\bibitem{BarHofKom16}
E.~Baroni, R.~v.~d. Hofstad, and J.~Komj{\'a}thy.
\newblock First passage percolation on random graphs with infinite variance
  degrees.
\newblock arXiv:1506.01255, 2016.

\bibitem{berndtsson1979exponential}
B.~Berndtsson and P.~Jagers.
\newblock Exponential growth of a branching process usually implies stable age
  distribution.
\newblock {\em Journal of Applied Probability}, pages 651--656, 1979.

\bibitem{Bha07}
S.~Bhamidi.
\newblock Universal techniques to analyze preferential attachment trees: Global
  and local analysis.
\newblock {\em {\tt http://www.unc.edu/~bhamidi/preferent.pdf}}, 2007.

\bibitem{BHH10}
S.~Bhamidi, R.~v.~d. Hofstad, and G.~Hooghiemstra.
\newblock First passage percolation on random graphs with finite mean degrees.
\newblock {\em Ann. Appl. Probab.}, 20(5):1907--1965, 2010.

\bibitem{BHH11}
S.~Bhamidi, R.~v.~d. Hofstad, and G.~Hooghiemstra.
\newblock First passage percolation on the {E}rd{\H o}s-{R}{\'e}nyi random
  graph.
\newblock {\em Combinatorics, Probability and Computing}, 20:683--707, 2011.

\bibitem{BHH14}
S.~Bhamidi, R.~v.~d. Hofstad, and G.~Hooghiemstra.
\newblock Universality for first passage percolation on sparse random graphs.
\newblock arXiv:1210.6839, 2014.

\bibitem{bhahofkom}
S.~Bhamidi, R.~Van Der~Hofstad, J.~Komj{\'a}thy, et~al.
\newblock The front of the epidemic spread and first passage percolation.
\newblock {\em Journal of Applied Probability}, 51:101--121, 2014.

\bibitem{BIg90}
J.~Biggins.
\newblock The central limit theorem for the supercritical branching random
  walk, and related results.
\newblock {\em Stochastic Processes and their Applications}, 34(2):255 -- 274,
  1990.

\bibitem{BiGoTe87}
N.~H. Bingham, C.~M. Goldie, and J.~L. Teugels.
\newblock {\em Regular Variation (Encyclopedia of Mathematics and its
  Applications)}.
\newblock Cambridge University Press, jun 1987.

\bibitem{Bollobas01}
B.~Bollob{\'a}s.
\newblock {\em Random Graphs}.
\newblock Cambridge University Press, 2001.

\bibitem{BJR}
B.~Bollob{\'a}s, S.~Janson, and O.~Riordan.
\newblock The phase transition in inhomogeneous random graphs.
\newblock {\em Random Struct. Algorithms}, 31(1):3--122, Aug. 2007.

\bibitem{born02}
S.~Bornholdt and H.~G. Schuster.
\newblock {\em Handbook of Graphs \& Networks}.
\newblock Wiley Online Library, 2002.

\bibitem{bramson78}
M.~D. Bramson.
\newblock Minimal displacement of branching random walk.
\newblock {\em Zeitschrift f{\"u}r Wahrscheinlichkeitstheorie und verwandte
  Gebiete}, 45(2):89--108, 1978.

\bibitem{Bu71}
W.~B\"{u}hler.
\newblock Generations and degree of relationship in supercritical markov
  branching processes.
\newblock {\em Probability Theory and Related Fields}, 18(2):141--152, 1971.

\bibitem{Bu72}
W.~B{\"u}hler.
\newblock The distribution of generations and other aspects of the family
  structure of branching processes.
\newblock In {\em Proc. 6th Berkeley Symp. Math. Statist. Prob}, volume~3,
  pages 463--480, 1972.

\bibitem{Bu74}
W.~B{\"u}hler.
\newblock On the family structure of populations.
\newblock {\em Advances in Applied Probability}, 6(2):192--193, 1974.

\bibitem{Dav78}
P.~Davies.
\newblock The simple branching process: a note on convergence when the mean is
  infinite.
\newblock {\em Journal of Applied Probability}, 15(3):466, 1978.

\bibitem{DeHo91}
F.~Dekking and B.~Host.
\newblock Limit distributions for minimal displacement of branching random
  walks.
\newblock {\em Probability Theory and Related Fields}, 90(3):403--426, 1991.

\bibitem{Der16}
S.~Dereich.
\newblock Preferential attachment with fitness: unfolding the condensate.
\newblock {\em Electronic Journal of Probability}, 2016.

\bibitem{DerMor16}
S.~Dereich, P.~M{\"o}rters, and C.~Mailler.
\newblock Non-extensive condensation in reinforced branching processes.
\newblock arXiv:1601.08128, 2016.

\bibitem{ER60}
P.~Erd{\H o}s and A.~R\'enyi.
\newblock On the evolution of random graphs.
\newblock {\em Publication of the Mathematical Institute of the Hungarian
  Academy of Sciences}, pages 17--61, 1960.

\bibitem{falo99}
M.~Faloutsos, P.~Faloutsos, and C.~Faloutsos.
\newblock On power-law relationships of the internet topology.
\newblock In {\em ACM SIGCOMM computer communication review}, volume~29, pages
  251--262. ACM, 1999.

\bibitem{Grey74}
D.~Grey.
\newblock Explosiveness of age-dependent branching processes.
\newblock {\em Zeitschrift f\"{u}r Wahrscheinlichkeitstheorie und Verwandte
  Gebiete}, 28(2):129--137, 1974.

\bibitem{Gul14}
L.~Gulikers.
\newblock Explosiveness of age-dependent branching processes with contagious
  and incubation periods.
\newblock arXiv:1510.03193, 2014.

\bibitem{Ham74}
J.~M. Hammersley.
\newblock Postulates for subadditive processes.
\newblock {\em The Annals of Probability}, 2(4):652--680, 08 1974.

\bibitem{Har63}
T.~E. Harris.
\newblock {\em The Theory of Branching Processes.}
\newblock Springer, 1963.

\bibitem{H10}
R.~v.~d. Hofstad.
\newblock {\em Random Graphs and Complex Networks}.
\newblock Springer, 2010.
\newblock book in preparation.

\bibitem{HHM05}
R.~v.~d. Hofstad, G.~Hooghiemstra, and P.~van Mieghem.
\newblock Distances in random graphs with finite variance degrees.
\newblock {\em Random Structures and Algorithms}, 27:76--123, 2005.

\bibitem{HHZ07}
R.~v.~d. Hofstad, G.~Hooghiemstra, and D.~Znamenski.
\newblock Distances in random graphs with finite mean and infinite variance
  degrees.
\newblock {\em Electron. J. Probab.}, 12:no. 25, 703--766, 2007.

\bibitem{HofKom15}
R.~v.~d. Hofstad and J.~Komj{\'a}thy.
\newblock Fixed speed competition on the configuration model with infinite
  variance degrees: equal speeds.
\newblock arXiv:1503.09046 [math.PR], 2015.

\bibitem{jagers1974convergence}
P.~Jagers.
\newblock Convergence of general branching processes and functionals thereof.
\newblock {\em Journal of Applied Probability}, pages 471--478, 1974.

\bibitem{jagers1984growth}
P.~Jagers and O.~Nerman.
\newblock The growth and composition of branching populations.
\newblock {\em Advances in Applied Probability}, pages 221--259, 1984.

\bibitem{kalu10}
P.~Kaluza, A.~K{\"o}lzsch, M.~T. Gastner, and B.~Blasius.
\newblock The complex network of global cargo ship movements.
\newblock {\em Journal of the Royal Society Interface}, 7(48):1093--1103, 2010.

\bibitem{Kar30}
J.~Karamata.
\newblock Sur un mode de croissance reguliere des fonctions, mathematica (cluj)
  4 (1930), 38-53.
\newblock {\em Mathematica (Cluj)}, 4:38--53, 1930.

\bibitem{Khar1}
B.~P. Kharlamov.
\newblock The numbers of generations in a branching process with an arbitrary
  set of particle types.
\newblock {\em Theory of Probability and Its Applications}, 14(3):432--449,
  1969.

\bibitem{Kin75}
J.~F.~C. Kingman.
\newblock The first birth problem for an age-dependent branching process.
\newblock {\em The Annals of Probability}, 3(5):790--801, 10 1975.

\bibitem{KolKom12}
I.~{Kolossv{\'a}ry} and J.~{Komj{\'a}thy}.
\newblock First passage percolation on inhomogeneous random graphs.
\newblock {\em Advances of Applied Probability}, 47, June 2015.

\bibitem{KomVad15}
J.~Komj{\'a}thy and V.~Vadon.
\newblock First passage percolation on the newman--watts small world model.
\newblock {\em Journal of Statistical Physics}, pages 1--35, 2016.

\bibitem{MolRee95}
M.~Molloy and B.~Reed.
\newblock A critical point for random graphs with a given degree sequence.
\newblock {\em Random Structures and Algorithms}, 6(2-3):161--180, 1995.

\bibitem{MolRee98}
M.~Molloy and B.~Reed.
\newblock The size of the giant component of a random graph with a given degree
  sequence.
\newblock {\em Comb. Probab. Comput.}, 7(3):295--305, Sept. 1998.

\bibitem{nerman1981convergence}
O.~Nerman.
\newblock On the convergence of supercritical general (cmj) branching
  processes.
\newblock {\em Probability Theory and Related Fields}, 57(3):365--395, 1981.

\bibitem{newman03}
M.~E. Newman.
\newblock The structure and function of complex networks.
\newblock {\em SIAM review}, 45(2):167--256, 2003.

\bibitem{radi12}
F.~Radicchi, S.~Fortunato, and A.~Vespignani.
\newblock Citation networks.
\newblock In {\em Models of science dynamics}, pages 233--257. Springer, 2012.

\bibitem{Rudas08}
A.~Rudas and B.~T{\'o}th.
\newblock Random tree growth with branching processesÑa survey.
\newblock In {\em Handbook of large-scale random networks}, pages 171--202.
  Springer, 2008.

\bibitem{Rudas07}
A.~Rudas, B.~T{\'o}th, and B.~Valk{\'o}.
\newblock Random trees and general branching processes.
\newblock {\em Random Structures \& Algorithms}, 31(2):186--202, 2007.

\bibitem{Sag15}
S.~Sagitov.
\newblock Tail generating functions for the extendable branching processes.
\newblock arXiv:1511.05407, 2015.

\bibitem{SagLin15}
S.~Sagitov and A.~Lindo.
\newblock A special family of galton-watson processes with explosions.
\newblock arXiv:1502.07538, 2015.

\bibitem{Sa71}
M.~Samuels.
\newblock Distribution of the branching-process population among generations.
\newblock {\em Journal of Applied Probability}, pages 655--667, 1971.

\bibitem{Sav69}
T.~H. Savits et~al.
\newblock The explosion problem for branching {M}arkov process.
\newblock {\em Osaka Journal of Mathematics}, 6(2):375--395, 1969.

\bibitem{BarSch77}
H.-J. Schuh and A.~Barbour.
\newblock On the asymptotic behaviour of branching processes with infinite
  mean.
\newblock {\em Advances in Applied Probability}, pages 681--723, 1977.

\bibitem{Sev70}
B.~Sevast'yanov.
\newblock Necessary condition for the regularity of branching processes.
\newblock {\em Mathematical notes of the Academy of Sciences of the USSR},
  7(4):234--238, 1970.

\bibitem{Sev67}
B.~A. Sevast'yanov.
\newblock On the regularity of branching processes.
\newblock {\em Mathematical Notes}, 1(1):34--40, 1967.

\bibitem{zika}
{The Guardian}.
\newblock Zika virus spreading explosively, says {W}orld {H}ealth
  {O}rganisation.
\newblock
  \url{http://www.theguardian.com/world/2016/jan/28/zika-virus-spreading-explosively-says-world-health-organisation}.
\newblock Accessed 28 Jan 2016.

\bibitem{ebola}
{The Washington Post}.
\newblock Ebola spreads slower, kills more than other diseases.
\newblock
  \url{http://www.washingtonpost.com/wp-srv/special/health/how-ebola-spreads/}.
\newblock Accessed 4 Feb 2016.

\bibitem{Vat87}
V.~Vatutin.
\newblock Sufficient conditions for regularity of bellman harris branching
  processes.
\newblock {\em Theory of Probability and Its Applications}, 31(1):50--57, 1987.

\end{thebibliography}

\end{document}